\documentclass{article}
\usepackage[utf8]{inputenc} % allow utf-8 input
\usepackage[T1]{fontenc}    % use 8-bit T1 fonts
\usepackage{fullpage}
\usepackage[numbers,compress]{natbib}

\usepackage{url}            % simple URL typesetting
\usepackage{booktabs}       % professional-quality tables
\usepackage{amsfonts}       % blackboard math symbols
\usepackage{microtype}      % microtypography
\usepackage{wrapfig}
\usepackage{caption}
\usepackage{nicefrac}       % compact symbols for 1/2, etc.

\usepackage{times}
\usepackage{algorithm}
\usepackage{adjustbox}
\usepackage[noend]{algpseudocode}
\usepackage{amsmath}
\usepackage{amssymb}
\usepackage{amsthm}
\usepackage{bm}
\usepackage{graphicx}
\usepackage{latexsym}
\usepackage{mathtools}
\usepackage{multirow}
\usepackage{paralist}
\usepackage{minitoc} % for appendix toc
\usepackage{xspace}
\usepackage{enumitem}
\usepackage{dsfont}

\usepackage[colorlinks,citecolor=bluegray,linkcolor=darkbrown,urlcolor=blue,breaklinks]{hyperref}
\usepackage{cleveref}

\newtheorem{theorem}{Theorem}
\newtheorem{lemma}[theorem]{Lemma}
\newtheorem{proposition}[theorem]{Proposition}
\newtheorem{corollary}[theorem]{Corollary}
\newtheorem{remark}{Remark}
\theoremstyle{definition}
\newtheorem{definition}{Definition}
\newtheorem{assumption}{Assumption}
\newtheorem{example}{Example}

\newtheorem{innercustomasmp}{Assumption}
\newenvironment{customasmp}[1]
  {\renewcommand\theinnercustomasmp{#1}\innercustomasmp}
  {\endinnercustomasmp}
  
\newtheorem{innercustomthm}{Theorem}
\newenvironment{customthm}[1]
  {\renewcommand\theinnercustomthm{#1}\innercustomthm}
  {\endinnercustomthm}
  
\newtheorem{innercustomprop}{Proposition}
\newenvironment{customprop}[1]
  {\renewcommand\theinnercustomprop{#1}\innercustomprop}
  {\endinnercustomprop}

\newcommand{\myparagraph}[1]{\par\noindent\textbf{{#1}.}} % For conference formats

\usepackage[dvipsnames]{xcolor} % color names 
\usepackage{tikz}
\usepackage{pgfplots}
\usetikzlibrary{shapes}
\usetikzlibrary{positioning}
\usetikzlibrary{plotmarks}
\usetikzlibrary{patterns}
\usetikzlibrary{intersections,shapes.arrows}
\usetikzlibrary{pgfplots.fillbetween}
\definecolor{darkpink}{rgb}{0.91, 0.33, 0.5}

\definecolor{puorange}{rgb}{0.80,0.20,0}
\definecolor{bluegray}{rgb}{0.04,0,0.7}
\definecolor{greengray}{rgb}{0.05,0.50,0.15}
\definecolor{darkbrown}{rgb}{0.40,0.2,0.05}
\definecolor{darkcyan}{rgb}{0,0.4,1}
\definecolor{black}{rgb}{0,0,0}
\definecolor{grey}{rgb}{0.93,0.93,0.93}
\definecolor{royalazure}{rgb}{0.0, 0.22, 0.66}

\crefname{section}{Sec.}{Sections}
\crefname{appendix}{Appx.}{Appxs}

\crefname{theorem}{Thm.}{Thms.}
\crefname{innercustomthm}{Thm.}{Thms.}
\crefname{lemma}{Lem.}{Lems.}
\crefname{corollary}{Cor.}{Cors.}
\crefname{proposition}{Prop.}{Props.}
\crefname{innercustomprop}{Prop.}{Props.}
\crefname{assumption}{Asm.}{Asms.}
\crefname{innercustomasmp}{Asmp.}{Asmps.}
\Crefname{example}{Ex.}{Exs.}

\crefname{algorithm}{Alg.}{Algs.}
\Crefname{algorithm}{Algorithm}{Algorithms}
\crefname{figure}{Fig.}{Figs.}
\crefname{table}{Tab.}{Tabs.}
\newcommand{\score}{\ell}
\newcommand{\grad}{S}
\newcommand{\risk}{L}
\newcommand{\rao}{T_{\text{Rao}}}
\newcommand{\lr}{T_{\text{LR}}}
\newcommand{\wald}{T_{\text{Wald}}}
\newcommand{\bbar}[1]{\bar{\bar{#1}}}  %% useful macros
\newcommand{\reals}{{\mathbb R}}

\newcommand{\dom}{\operatorname{dom}}

\newcommand{\abs}[1]{\left| #1 \right|}
\newcommand{\norm}[1]{\left\lVert #1 \right\rVert}
\newcommand{\anorm}[1]{\left| #1 \right|}
\newcommand{\ones}{\operatorname{\mathbf 1}}

\newcommand{\Rank}{\operatorname{\bf Rank}}
\newcommand{\Tr}{\operatorname{\bf Tr}}
\newcommand{\diag}{\operatorname{\bf diag}}

\newcommand{\ip}[1]{{\langle #1 \rangle}}

\newcommand{\Expect}{\operatorname{\mathbb E}}
\newcommand{\Var}{\operatorname{\mathbb{V}ar}}

\newcommand{\Prob}{\operatorname{\mathbb P}}

\newcommand \D {\mathrm{d}}
\newcommand{\kl}{\operatorname{KL}}

\newcommand{\hnull}{\mathbf{H}_0}
\newcommand{\halt}{\mathbf{H}_1}

\DeclareMathOperator*{\argmin}{arg\,min}

\newcommand{\calA}{\mathcal{A}}
\newcommand{\calB}{\mathcal{B}}
\newcommand{\calC}{\mathcal{C}}
\newcommand{\calD}{\mathcal{D}}
\newcommand{\calE}{\mathcal{E}}
\newcommand{\calF}{\mathcal{F}}

\newcommand{\calL}{\mathcal{L}}

\newcommand{\calN}{\mathcal{N}}
\newcommand{\calP}{\mathcal{P}}

\newcommand{\calV}{\mathcal{V}}

\newcommand{\calX}{\mathcal{X}}
\newcommand{\calY}{\mathcal{Y}}
\newcommand{\calZ}{\mathcal{Z}}

\newcommand{\bbZ}{\mathbb{Z}}

\newcommand{\ind}{\mathds{1}}

\newcommand{\txtover}[2]{\overset{\mbox{\scriptsize #1}}{#2}}

\title{Confidence Sets under Generalized Self-Concordance}

\author{Lang Liu \qquad Zaid Harchaoui \\
Department of Statistics, University of Washington}
\date{}

\begin{document}

\maketitle
\doparttoc % Tell to minitoc to generate a toc for the parts
\faketableofcontents % Run a fake tableofcontents command for the partocs

\begin{abstract}
    This paper revisits a fundamental problem in statistical inference from a non-asymptotic theoretical viewpoint---the construction of confidence sets. We establish a finite-sample bound for the estimator, characterizing its asymptotic behavior in a non-asymptotic fashion. An important feature of our bound is that its dimension dependency is captured by the effective dimension---the trace of the limiting sandwich covariance---which can be much smaller than the parameter dimension in some regimes. We then illustrate how the bound can be used to obtain a confidence set whose shape is adapted to the optimization landscape induced by the loss function. Unlike previous works that rely heavily on the strong convexity of the loss function, we only assume the Hessian is lower bounded at optimum and allow it to gradually becomes degenerate. This property is formalized by the notion of generalized self-concordance which originated from convex optimization. Moreover, we demonstrate how the effective dimension can be estimated from data and characterize its estimation accuracy. We apply our results to maximum likelihood estimation with generalized linear models, score matching with exponential families, and hypothesis testing with Rao's score test.
\end{abstract}

\section{INTRODUCTION}
\label{sec:intro}
The problem of statistical inference on learned parameters is regaining the importance it deserves
as machine learning and data science are increasingly impacting humanity and society through an increasingly large range of successful applications from transportation to healthcare \citep[see, e.g.,][]{fan2020statistical,efron2021computer}.
The classical asymptotic theory of M-estimation is well established in a rather general setting under the assumption that the parametric model is well-specified, i.e., the underlying data distribution belongs to the parametric family.
Two types of confidence sets can be constructed from this theory: (a) the Wald-type one which relies on the weighted difference between the estimator and the target parameter, and (b) the likelihood-ratio-type one based on the log-likelihood ratio between the estimator and the target parameter.
The main tool is the local asymptotic normality (LAN) condition introduced by \citet{lecam1960locally}.
We mention here, among many of them, the monographs \citep{ibragimov1981statistical,van2000asymptotic,geer2000empirical}.

In many real problems, the parametric model is usually an approximation to the data distribution, so it is too restrictive to assume that the model is well-specified.
To relax this restriction, model misspecification has been considered in the asymptotic regime; see, e.g., \citep{huber1967under,wakefield2013bayesian,dawid2016scoring}.
Another limitation of classical asymptotic theory is its asymptotic regime where $n \rightarrow \infty$ and the parameter dimension $d$ is fixed.
This is inapplicable in the modern context where the data are of a rather high dimension involving a huge number of parameters.

The non-asymptotic viewpoint has been fruitful to address high dimensional problems---the results are developed for all fixed $n$ so that it also captures the asymptotic regime where $d$ grows with $n$.
Early works in this line of research focus on specific models such as Gaussian models~\citep{beran1996confidence,beran1998modulation,laurent2000adaptive,baraud2004confidence}, ridge regression~\citep{hsu2012random}, logistic regression \citep{bach2010self}, and robust M-estimation~\citep{zhou2018huber,chen2020robust}; see~\citet{bach2021learning} for a survey.~\citet{spokoiny2012parametric} addressed the finite-sample regime in full generality in a spirit similar to the classical LAN theory.
The approach of~\cite{spokoiny2012parametric} relies on heavy empirical process machinery and requires strong global assumptions on the deviation of the empirical risk process. More recently,~\citet{ostrovskii2021finite} focused on risk bounds, specializing their discussion to linear models with (pseudo) self-concordant losses and obtained a more transparent analysis under neater assumptions.

\begin{table*}[t]
    \caption{Loss function of generalized linear models.}
    \label{tab:glms}
    \centering
    \renewcommand{\arraystretch}{1.2}
    \begin{tabular}{lccc}
        \addlinespace[0.4em]
        \toprule
        & \multicolumn{1}{c}{\textbf{Data}} & \multicolumn{1}{c}{\textbf{Model}} & \textbf{Loss} \\
        \midrule
        Linear & $(X, Y)$ & $Y \mid X \sim \mathcal{N}(\theta^\top X, \sigma^2)$ & $\frac12 (y - \theta^\top x)^2$ \\
        Logistic & $(X, Y)$ & $Y \mid X \sim \mbox{Bernoulli}\big((1 + e^{-\theta^\top X})^{-1}\big)$ & $\log{\big(1 + \exp(-y(\theta^\top x))\big)}$\\
        Poisson & $(X, Y)$ & $Y \mid X \sim \mbox{Poisson}(\exp(\theta^\top X))$ & $-y(\theta^\top x) + \exp(\theta^\top x)$ \\
        \bottomrule
    \end{tabular}
\end{table*}

A critical tool arising from this line of research is the so-called \emph{Dikin ellipsoid}, a geometric object identified in the theory of convex optimization~\citep{yurii1994interior,ben2001lectures,boyd2004convex,tunccel2010self,bubeck2016black,bubeck:2019}.
The Dikin ellipsoid corresponds to the distance measured by the Euclidean distance weighted by the Hessian matrix at the optimum.
This weighted Euclidean distance is adapted to the geometry near the target parameter and thus leads to sharper bounds that do not depend on the minimum eigenvalue of the Hessian. This important property has been used fruitfully in various problems of learning theory and mathematical statistics \citep{zhang2015disco,yang2016optimistic,faury2020improved}.

\myparagraph{Outline}
We review in \Cref{sec:problem} the empirical risk minimization framework and the two types of confidence sets from classical asymptotic theory.
We establish finite-sample bounds to characterize these two confidences sets, whose sizes are controlled by the \emph{effective dimension}, in a non-asymptotic fashion in \Cref{sec:main_results}.
Our results hold for a general class of models characterized by the notion of \emph{generalized self-concordance}.
Along the way, we show how the effective dimension can be estimated from data and provide its estimation accuracy.
This is a novel result and is of independent interest.
We apply our results to compare Rao's score test, the likelihood ratio test, and the Wald test for goodness-of-fit testing in \Cref{sec:application}.
Finally, in \Cref{sec:experiments}, we illustrate the interest of our results on synthetic data.

\section{PROBLEM FORMULATION}
\label{sec:problem}
We briefly recall the framework of statistical inference via empirical risk minimization.
Let $(\bbZ, \calZ)$ be a measurable space.
Let $Z \in \bbZ$ be a random element following some unknown distribution $\Prob$.
Consider a parametric family of distributions $\calP_\Theta := \{P_\theta: \theta \in \Theta \subset \reals^d\}$ which may or may not contain $\Prob$.
We are interested in finding the parameter $\theta_\star$ so that the model $P_{\theta_\star}$ best approximates the underlying distribution $\Prob$.
For this purpose, we choose a \emph{loss function} $\score$ and minimize the \emph{population risk} $\risk(\theta) := \Expect_{Z \sim \Prob}[\score(\theta; Z)]$.
Throughout this paper, we assume that
\begin{align*}
     \theta_\star = \argmin_{\theta \in \Theta} L(\theta)
\end{align*}
uniquely exists and satisfies $\theta_\star \in \text{int}(\Theta)$, $\nabla_\theta L(\theta_\star) = 0$, and $\nabla_\theta^2 L(\theta_\star) \succ 0$.

\myparagraph{Consistent loss function}
We focus on loss functions that are consistent in the following sense.

\begin{customasmp}{0}\label{asmp:proper_loss}
    When the model is \emph{well-specified}, i.e., there exists $\theta_0 \in \Theta$ such that $\Prob = P_{\theta_0}$, it holds that $\theta_0 = \theta_\star$.
    We say such a loss function is \emph{consistent}.
\end{customasmp}

In the statistics literature, such loss functions are known as proper scoring rules \citep{dawid2016scoring}.
We give below two popular choices of consistent loss functions.

\begin{example}[Maximum likelihood estimation]
    A widely used loss function in statistical machine learning is the negative log-likelihood $\score(\theta; z) := -\log{p_\theta(z)}$ where $p_\theta$ is the probability mass/density function for the discrete/continuous case.
    When $\Prob = P_{\theta_0}$ for some $\theta_0 \in \Theta$,
    we have $L(\theta) = \Expect[-\log{p_\theta(Z)}] = \kl(p_{\theta_0} \Vert p_\theta) - \Expect[\log{p_{\theta_0}(Z)}]$ where $\kl$ is the Kullback-Leibler divergence.
    As a result, $\theta_0 \in \argmin_{\theta \in \Theta} \kl(p_{\theta_0} \Vert p_\theta) = \argmin_{\theta \in \Theta} L(\theta)$.
    Moreover, if there is no $\theta$ such that $p_\theta \txtover{a.s.}{=} p_{\theta_0}$, then $\theta_0$ is the unique minimizer of $L$.
    We give in \Cref{tab:glms} a few examples from the class of generalized linear models (GLMs) proposed by \citet{nelder1972generalized}.
\end{example}

\begin{example}[Score matching estimation]
    Another important example appears in \emph{score matching} \citep{hyvarinen2005estimation}.
    Let $\bbZ = \reals^\tau$.
    Assume that $\Prob$ and $P_\theta$ have densities $p$ and $p_\theta$ w.r.t the Lebesgue measure, respectively.
    Let $p_\theta(z) = q_\theta(z) / \Lambda(\theta)$ where $\Lambda(\theta)$ is an unknown normalizing constant. We can choose the loss
    \begin{align*}
        \score(\theta; z) := \Delta_z \log{q_\theta(z)} + \frac12 \norm{\nabla_z \log{q_\theta(z)}}^2 + \text{const}.
    \end{align*}
    Here $\Delta_z := \sum_{k=1}^p \partial^2/\partial z_k^2$ is the Laplace operator.
    Since \cite[Thm.~1]{hyvarinen2005estimation}
    \begin{align*}
        L(\theta) = \frac12 \Expect\left[ \norm{\nabla_z q_\theta(z) - \nabla_z p(z)}^2 \right],
    \end{align*}
    we have, when $p = p_{\theta_0}$, that $\theta_0 \in \argmin_{\theta \in \Theta} L(\theta)$.
    In fact, when $q_\theta > 0$ and there is no $\theta$ such that $p_\theta \txtover{a.s.}{=} p_{\theta_0}$, the true parameter $\theta_0$ is the unique minimizer of $L$ \cite[Thm.~2]{hyvarinen2005estimation}.
\end{example}

\myparagraph{Empirical risk minimization}
Assume now that we have an i.i.d.~sample $\{Z_i\}_{i=1}^n$ from $\Prob$.
To learn the parameter $\theta_\star$ from the data, we minimize the empirical risk to obtain the \emph{empirical risk minimizer}
\begin{align*}
    \theta_n \in \argmin_{\theta \in \Theta} \left[ L_n(\theta) := \frac1n \sum_{i=1}^n \score(\theta; Z_i) \right].
\end{align*}
This applies to both maximum likelihood estimation and score matching estimation. 
In \Cref{sec:main_results}, we will prove that, with high probability, the estimator $\theta_n$ exists and is unique under a generalized self-concordance assumption.

\begin{figure}
    \centering
    \includegraphics[width=0.45\textwidth]{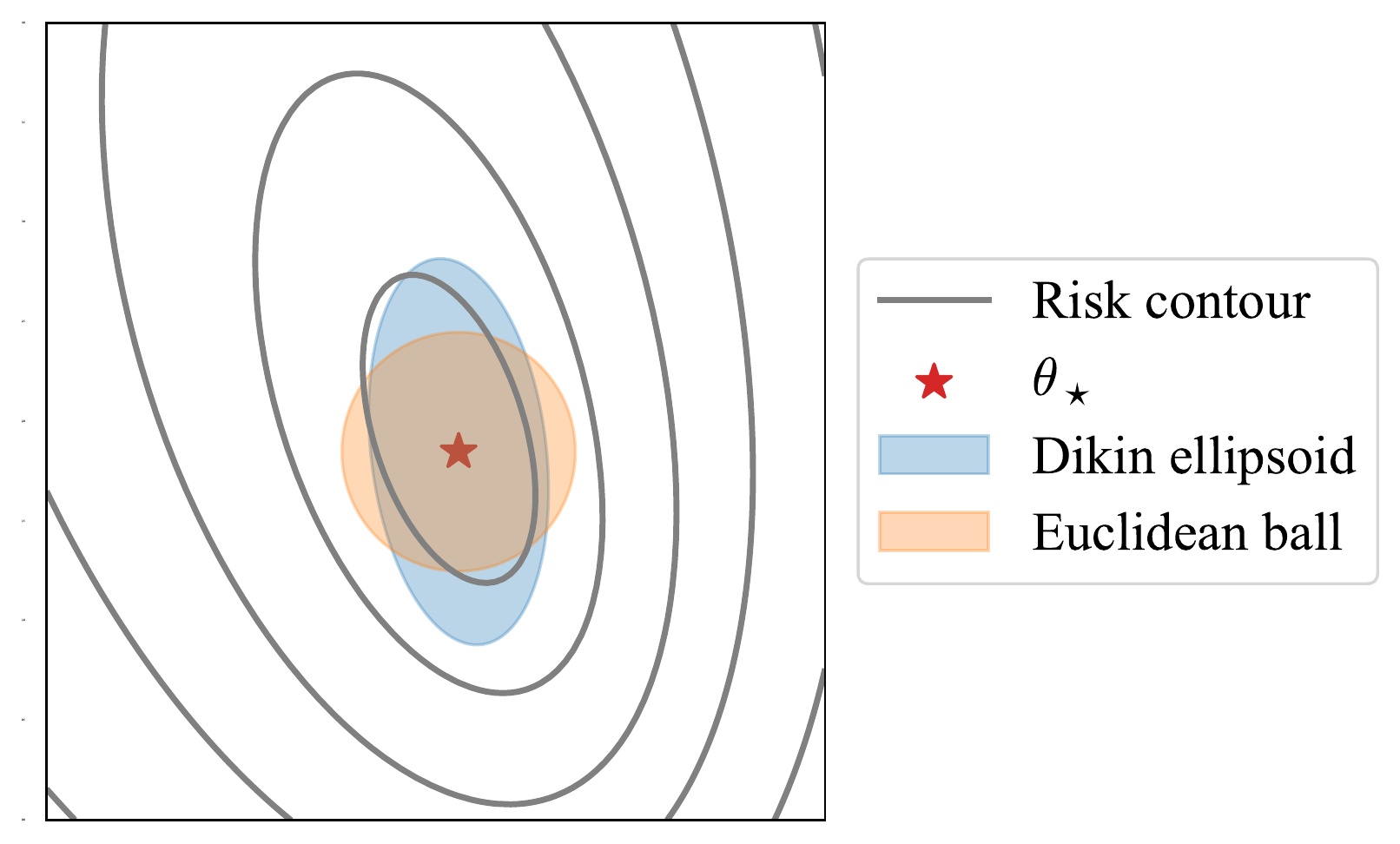} %0.4
    \caption{Dikin ellipsoid and Euclidean ball.}
    \label{fig:logistic_dikin}
\end{figure}

\myparagraph{Confidence set}
In statistical inference, it is of great interest to quantify the uncertainty in the estimator $\theta_n$.
In classical asymptotic theory, this is achieved by constructing an asymptotic confidence set.
We review here two commonly used ones, assuming the model is well-specified.
We start with the \emph{Wald confidence set}.
It holds that $n(\theta_n - \theta_\star)^\top H_n(\theta_n) (\theta_n - \theta_\star) \rightarrow_d \chi_d^2$, where $H_n(\theta) := \nabla^2 L_n(\theta)$.
Hence, one may consider a confidence set $\{\theta: n(\theta_n - \theta)^\top H_n(\theta_n) (\theta_n - \theta) \le q_{\chi_d^2}(\delta) \}$ where $q_{\chi_d^2}(\delta)$ is the upper $\delta$-quantile of $\chi_d^2$.
The other is the \emph{likelihood-ratio (LR) confidence set} constructed from the limit $2n [L_n(\theta_\star) - L_n(\theta_n)] \rightarrow_d \chi_d^2$, which is known as the Wilks' theorem \citep{wilks1938large}.
These confidence sets enjoy two merits: 1) their shapes are an ellipsoid (known as the \emph{Dikin ellipsoid}) which is adapted to the optimization landscape induced by the population risk; 2) they are asymptotically valid, i.e., their coverages are exactly $1 - \delta$ as $n \rightarrow \infty$.
However, due to their asymptotic nature, it is unclear how large $n$ should be in order for it to be valid.

Non-asymptotic theory usually focuses on developing finite-sample bounds for the \emph{excess risk}, i.e., $\Prob(L(\theta_n) - L(\theta_\star) \le C_n(\delta)) \ge 1 - \delta$.
To obtain a confidence set, one may assume that the population risk is twice continuously differentiable and $\lambda$-strongly convex.
Consequently, we have $\lambda \norm{\theta_n - \theta_\star}_2^2 / 2 \le L(\theta_n) - L(\theta_\star)$ and thus we can consider the confidence set $\calC_{\text{finite}, n}(\delta) := \{\theta: \norm{\theta_n - \theta}_2^2 \le 2C_n(\delta)/\lambda\}$.
Since it originates from a finite-sample bound, it is valid for fixed $n$, i.e., $\Prob(\theta_\star \in \calC_{\text{finite}, n}(\delta)) \ge 1 - \delta$ for all $n$; however, it is usually conservative, meaning that the coverage is strictly larger than $1 - \delta$.
Another drawback is that its shape is a Euclidean ball which remains the same no matter which loss function is chosen.
We illustrate this phenomenon in \Cref{fig:logistic_dikin}.
Note that a similar observation has also been made in the bandit literature \citep{faury2020improved}.

We are interested in developing finite-sample confidence sets.
However, instead of using excess risk bounds and strong convexity, we construct in \Cref{sec:main_results} the Wald and LR confidence sets in a non-asymptotic fashion, under a generalized self-concordance condition.
These confidence sets have the same shape as their asymptotic counterparts while maintaining validity for fixed $n$.
These new results are achieved by characterizing the critical sample size enough to enter the asymptotic regime.

\section{MAIN RESULTS}
\label{sec:main_results}
\begin{figure}
    \centering
    \includegraphics[width=0.6\textwidth]{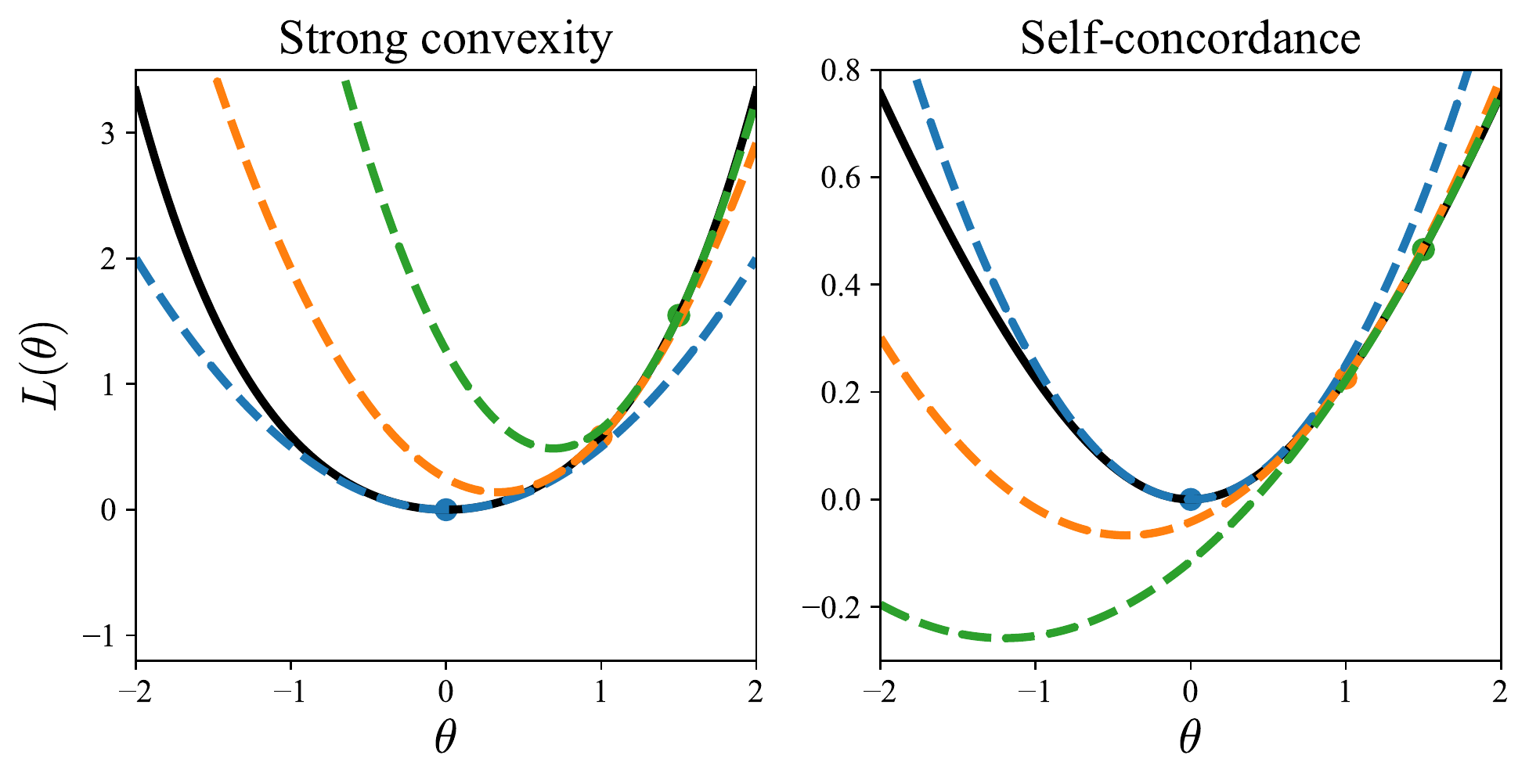} %0.45
    \caption{Strong convexity v.s.~self-concordance. Black curve: population risk; colored dot: reference point; colored dashed curve: quadratic approximation at the corresponding reference point.}
    \label{fig:convex_concordance}
\end{figure}

\subsection{Preliminaries}
\label{sub:preliminary}

\myparagraph{Notation}
We denote by $\grad(\theta; z) := \nabla_\theta \score(\theta; z)$ the gradient of the loss at $z$ and $H(\theta; z) := \nabla_\theta^2 \score(\theta; z)$ the Hessian at $z$.
Their population versions are $\grad(\theta) := \Expect[\grad(\theta; Z)]$ and $H(\theta) := \Expect[H(\theta; Z)]$, respectively.
We assume standard regularity assumptions so that $\grad(\theta) = \nabla_\theta L(\theta)$ and $H(\theta) = \nabla_\theta^2 L(\theta)$.
We write $H_\star := H(\theta_\star)$.
Note that the two optimality conditions then read $\grad(\theta_\star) = 0$ and $H_\star \succ 0$.
It follows that $\lambda_\star := \lambda_{\min}(H_\star) > 0$ and $\lambda^\star := \lambda_{\max}(H_\star) > 0$.
Furthermore, we let $G(\theta; z) := S(\theta; z) S(\theta; z)^\top$ and $G(\theta) := \Expect[\grad(\theta; Z)\grad(\theta; Z)^\top]$ be the autocorrelation matrices of the gradient.
We write $G_\star := G(\theta_\star)$.
We define their empirical quantities as $L_n(\theta) := n^{-1} \sum_{i=1}^n \score(\theta; Z_i)$, $\grad_n(\theta) := n^{-1} \sum_{i=1}^n \grad(\theta; Z_i)$, $H_n(\theta) := n^{-1} \sum_{i=1}^n H(\theta; Z_i)$, and $G_n(\theta) := n^{-1} \sum_{i=1}^n G(\theta; Z_i)$.
The first step of our analysis is to localize the estimator to a \emph{Dikin ellipsoid} at $\theta_\star$ of radius $r$, i.e.,
\begin{align*}
    \Theta_r(\theta_\star) := \left\{\theta \in \Theta: \norm{\theta - \theta_\star}_{H_\star} < r \right\},
\end{align*}
where, given a positive semi-definite matrix $J$, we let $\norm{x}_J := \norm{J^{1/2} x}_2 = \sqrt{x^\top J x}$.

\myparagraph{Effective dimension}
A quantity that plays a central role in our analysis is the \emph{effective dimension}.
\begin{definition}
\label{def:effective_dim}
    We define the effective dimension to be
    \begin{align}
        d_\star := \Tr( H_\star^{-1/2} G_\star H_\star^{-1/2} ).
    \end{align}
\end{definition}
The effective dimension appears recently in non-asymptotic analyses of (penalized) M-estimation; see, e.g., \citep{spokoiny2017penalized,ostrovskii2021finite}.
It better characterizes the complexity of the parameter space $\Theta$ than the parameter dimension $d$.
When the model is well-specified, it can be shown that $H_\star = G_\star$ and thus $d_\star = d$.
When the model is misspecified, it can be much smaller than $d$ depending on the spectra of $H_\star$ and $G_\star$.
Moreover, it is closely connected to classical asymptotic theory of M-estimation under model misspecification---it is the trace of the limiting covariance matrix of $\sqrt{n}H_n(\theta_n)^{1/2}(\theta_n - \theta_\star)$;
see \Cref{sub:discussion} for a thorough discussion.

\myparagraph{Generalized self-concordance}
We will use the notion of \emph{self-concordance} from convex optimization in our analysis.
Self-concordance originated from the analysis of the interior-point and Newton-type convex optimization methods \citep{yurii1994interior}.
It was later modified by \citet{bach2010self}, which we call the \emph{pseudo self-concordance}, to derive finite-sample bounds for the generalization properties of the logistic regression.
Recently, \citet{sun2019generalized} proposed the \emph{generalized self-concordance} which unifies these two notions.
For a function $f: \reals^d \to \reals$, we define $D_x f(x)[u] := \frac{\D}{\D t} f(x + tu) |_{t = 0}$, $D_x^2 f(x)[u, v] := D_x (D_x f(x)[u])[v]$ for $x, u, v \in \reals^d$, and $D_x^3 f(x)[u, v, w]$ similarly.
\begin{definition}[Generalized self-concordance]
\label{def:general_self_concordance}
    Let $\calX \subset \reals^d$ be open and $f: \calX \rightarrow \reals$ be a closed convex function.
    For $R > 0$ and $\nu > 0$, we say $f$ is $(R, \nu)$-generalized self-concordant on $\calX$ if
    \begin{align*}
        \abs{D_x^3 f(x) [u, u, v]} \le R \norm{u}_{\nabla^2 f(x)}^2 \norm{v}_{\nabla^2 f(x)}^{\nu-2} \norm{v}_2^{3-\nu}
    \end{align*}
    with the convention $0/0 = 0$ for the case $\nu < 2$ and $\nu > 3$.
    Recall that $\norm{u}_{\nabla^2 f(x)}^2 := u^\top \nabla^2 f(x) u$.
\end{definition}

\myparagraph{Remark}
When $\nu = 2$ and $\nu = 3$, this definition recovers the pseudo self-concordance and the standard self-concordance, respectively.

In contrast to strong convexity which imposes a gross lower bound on the Hessian, generalized self-concordance specifies the rate at which the Hessian can vary, leading to a finer control on the Hessian.
Concretely, it allows us to bound the Hessian in a neighborhood of $\theta_\star$ with the Hessian at $\theta_\star$, which is key to controlling $H_n(\theta_n)$.
We illustrate the difference between them in \Cref{fig:convex_concordance}.
As we will see in \Cref{sub:main_results}, thanks to the generalized self-concordance, we are able to remove the direct dependency on $\lambda_\star$ in our confidence set.
To the best of our knowledge, this is the first work extending classical results for M-estimation to generalized self-concordant losses.

\myparagraph{Concentration of Hessian}
One key result towards deriving our bounds is the concentration of empirical Hessian, i.e., $(1 - c_n(\delta))H(\theta) \preceq H_n(\theta) \preceq (1 + c_n(\delta)) H(\theta)$ with probability at least $1 - \delta$.
When the loss function is of the form $\ell(\theta; z) := \ell(y, \theta^\top x)$ (e.g., GLMs), the empirical Hessian reads $H_n(\theta) = n^{-1} \sum_{i=1}^n \ell''(Y_i, \theta^\top X_i) X_i X_i^\top$ where $\ell''(y, \bar y) := \D^2 \ell(y, \bar y) / \D \bar y^2$, which is of the form of a sample covariance.
Assuming $X$ to be sub-Gaussian, \citet{ostrovskii2021finite} obtained a concentration bound for $H_n(\theta_\star)$ with $c_n(\delta) = O(\sqrt{(d + \log{(1/\delta)})/n})$ via the concentration bound for sample covariance \citep[Thm.~5.39]{vershynin2010introduction}.
For general loss functions, such a special structure cannot be exploited.
We overcame this challenge by the matrix Bernstein inequality \citep[Thm.~6.17]{wainwright2019high}, obtaining a sharper concentration bound with $c_n(\delta) := O(\sqrt{\log{(d/\delta)}/n})$.
Note that the matrix Bernstein inequality has been used to control the empirical Hessian of kernel ridge regression with random features \citep[Prop.~6]{rudi2017generalization} and later extended to regularized empirical risk minimization \citep[Lem.~30]{marteau2019beyond}.
However, their results require the regularization parameter to be strictly positive (otherwise the bounds are vacuous) and the sample Hessian to be bounded.
On the contrary, our technique allows for zero regularization and unbounded Hessian as long as the Hessian satisfies a matrix Bernstein condition.
Moreover, combining generalized self-concordance with matrix Bernstein, we are able to show the concentration of $H_n(\theta_n)$ around $H_\star$ for general losses, which is itself a novel result.

\subsection{Assumptions}
\label{sub:assumption}

Our key assumption is the generalized self-concordance of the loss function.
\begin{assumption}[Generalized self-concordance]
\label{asmp:self_concordance}
    For any $z \in \calZ$, the scoring rule $\score(\cdot; z)$ is $(R, \nu)$-generalized self-concordant for some $R > 0$ and $\nu \ge 2$.
    Moreover, $\risk(\cdot)$ is also $(R, \nu)$-generalized self-concordant.
\end{assumption}

\myparagraph{Remark}
If $\score(\cdot; z)$ is generalized self-concordant with $\nu = 2$, so is $\risk(\cdot)$.

Many loss functions in statistical machine learning satisfy this assumption.
We give in \Cref{sub:examples} examples from generalized linear models and score matching.

In order to control the empirical gradient $\grad_n(\theta)$, we assume that the normalized gradient at $\theta_\star$ is sub-Gaussian.
\begin{assumption}[Sub-Gaussian gradient]
\label{asmp:sub_gaussian}
    There exists a constant $K_1 > 0$ such that the normalized gradient at $\theta_\star$ is sub-Gaussian with parameter $K_1$, i.e., $\lVert G_\star^{-1/2} \grad(\theta_\star; Z) \rVert_{\psi_2} \le K_1$.
    Here $\norm{\cdot}_{\psi_2}$ is the sub-Gaussian norm whose definition is recalled in \Cref{sec:tools}.
\end{assumption}

When the loss function is of the form $\ell(\theta; z) = \ell(y, \theta^\top x)$, we have $S(\theta; Z) = \ell'(Y, \theta^\top X) X$.
As a result, \Cref{asmp:sub_gaussian} holds true if (i) $\ell'(Y, \theta_\star^\top X)$ is sub-Gaussian and $X$ is bounded or (ii) $\ell'(Y, \theta_\star^\top X)$ is bounded and $X$ is sub-Gaussian.
For least squares with $\ell(y, \theta^\top x) = \frac12 (y - \theta^\top x)^2$, the derivative $\ell'(Y, \theta_\star^\top X) = \theta_\star^\top X - Y$ is the negative residual.
\Cref{asmp:sub_gaussian} is guaranteed if the residual is sub-Gaussian and $X$ is bounded.
For logistic regression with $\ell(y, \theta^\top x) = -\log{\sigma(y\cdot \theta^\top x)}$ where $\sigma(u) = (1 + e^{-u})^{-1}$, the derivative $\ell'(Y, \theta_\star^\top X) = [\sigma(Y \cdot \theta_\star^\top X) - 1]Y \in [-1, 1]$ is bounded.
Thus, \Cref{asmp:sub_gaussian} is guaranteed if $X$ is sub-Gaussian.

In order to control the empirical Hessian, we assume that the Hessian of the loss function satisfies the matrix Bernstein condition in a neighborhood of $\theta_\star$.

\begin{assumption}[Matrix Bernstein of Hessian]
\label{asmp:bernstein}
    There exist constants $K_2, r > 0$ such that, for any $\theta \in \Theta_{r}(\theta_\star)$, the standardized Hessian
    \begin{align*}
        H(\theta)^{-1/2} H(\theta; Z) H(\theta)^{-1/2} - I_d
    \end{align*}
    satisfies a Bernstein condition (defined in \Cref{sec:tools}) with parameter $K_2$. Moreover,
    \begin{align*}
        \sigma_H^2 := \sup_{\theta \in \Theta_{r}(\theta_\star)} \norm{\Var\left( H(\theta)^{-\frac12}H(\theta; Z)H(\theta)^{-\frac12} \right)}_2 < \infty,
    \end{align*}
    where $\norm{\cdot}_2$ is the spectral norm and $\Var(J) := \Expect[JJ^\top] - \Expect[J] \Expect[J]^\top$.
    By convention, we let $\Theta_0(\theta_\star) = \{\theta_\star\}$.
\end{assumption}

\subsection{Main Results}
\label{sub:main_results}

We now give simplified versions of our main theorems.
We use $C_\nu$ to represent a constant depending only on $\nu$ that may change from line to line; and $C_{K_1, \nu}$ similarly.
We use $\lesssim$ and $\gtrsim$ to hide constants depending only on $K_1, K_2, \sigma_H, \nu$.
The precise versions can be found in \Cref{sec:proofs}.
Recall that $\lambda_\star := \lambda_{\min}(H_\star)$ and $\lambda^\star := \lambda_{\max}(H_\star)$.
\begin{theorem}\label{thm:risk_bound_generalized}
    Let $\nu \in [2, 3)$.
    Under \Cref{asmp:self_concordance,asmp:sub_gaussian,asmp:bernstein} with $r = 0$, it holds that,
    whenever
    \begin{align*}
        n \gtrsim \log{(2d/\delta)} + \lambda_\star^{-1} \left[ R^2 d_\star \log{(e/\delta)} \right]^{1/(3-\nu)},
    \end{align*}
    the empirical risk minimizer $\theta_n$ uniquely exists and satisfies, with probability at least $1 - \delta$,
    \begin{align}\label{eq:conf_bound}
        \norm{\theta_n - \theta_\star}^2_{H_\star} \lesssim \log{(e/\delta)} \frac{d_\star}{n}.
    \end{align}
\end{theorem}

With a local matrix Bernstein condition, we can replace $H_\star$ by $H_n(\theta_n)$ in \eqref{eq:conf_bound} and obtain a finite-sample version of the Wald confidence set.
\begin{theorem}\label{thm:conf_set}
    Let $\nu \in [2, 3)$.
    Suppose the same assumptions in \Cref{thm:risk_bound_generalized} hold true.
    Furthermore, suppose that \Cref{asmp:bernstein} holds with $r = C_\nu \lambda_\star^{(3-\nu)/2} / R$.
    Let $\calC_{\text{Wald}, n}(\delta)$ be
    \begin{align}\label{eq:my_conf_set}
        \left\{\theta \in \Theta: \norm{\theta - \theta_n}_{H_n(\theta_n)}^2 \le C_{K_1,\nu} \frac{d_\star}{n} \log{\frac{e}{\delta}} \right\}.
    \end{align}
    Then we have $\Prob(\theta_\star \in \calC_{\text{Wald}, n}(\delta)) \ge 1 - \delta$ whenever
    \begin{align}\label{eq:n_large_enough}
        n \gtrsim \log{\frac{2d}{\delta}} + d\log{n} + \lambda_\star^{-1}\left[ R^2 d_\star \log{\frac{e}{\delta}} \right]^{\frac1{3-\nu}}.
    \end{align}
\end{theorem}

\myparagraph{Remark}
In the precise versions of \Cref{thm:risk_bound_generalized,thm:conf_set}, the term $d_\star \log{(e/\delta)}$ in the bounds \eqref{eq:conf_bound} and \eqref{eq:my_conf_set} should be replaced by $d_\star + \log{(e/\delta)} \lVert G_\star^{1/2} H_\star^{-1} G_\star^{1/2} \rVert_2$, which almost match the misspecified Cram\'er-Rao lower bound \citep[e.g.,][Thm.~1]{fortunati2016misspecified} up to a constant factor.

\Cref{thm:conf_set} suggests that the tail probability of $\norm{\theta_n - \theta_\star}_{H_n(\theta_n)}^2$ is governed by a $\chi^2$ distribution with $d_\star$ degrees of freedom, which coincides with the asymptotic result.
In fact, according to \citet{huber1967under}, under suitable regularity assumptions, it holds that $\sqrt{n} H_n(\theta_n)^{1/2}(\theta_n - \theta_\star) \rightarrow_d W \sim \mathcal{N}(0, H_\star^{-1/2} G_\star H_\star^{-1/2})$ which implies that
\begin{align*}
    n(\theta_n - \theta_\star)^\top H_n(\theta_n) (\theta_n - \theta_\star) \rightarrow_d W^\top W.
\end{align*}
This induces an asymptotic confidence set with a similar form of \eqref{eq:my_conf_set} and radius $O(\Expect[W^\top W] / n) = O(d_\star / n)$.
Our result characterizes the \emph{critical sample size} enough to enter the asymptotic regime.

From \Cref{thm:conf_set} we can also derive a finite-sample version of the LR confidence set.
\begin{corollary}\label{cor:lr_conf_set}
    Let $\nu \in [2, 3)$.
    Suppose the same assumptions in \Cref{thm:conf_set} hold true.
    Let $\calC_{\text{LR}, n}(\delta)$ be
    \begin{align}\label{eq:lr_conf_set}
        \left\{\theta \in \Theta: 2[L_n(\theta) - L_n(\theta_n)] \le C_{K_1,\nu} \frac{d_\star}{n} \log{\frac{e}{\delta}} \right\}.
    \end{align}
    Then we have $\Prob(\theta_\star \in \calC_{\text{LR}, n}(\delta)) \ge 1 - \delta$ whenever
    \begin{align*}
        n \gtrsim \log{\frac{2d}{\delta}} + d\log{n} + \lambda_\star^{-1}\left[ R^2 d_\star \log{\frac{e}{\delta}} \right]^{\frac1{3-\nu}}.
    \end{align*}
\end{corollary}

We give the proof sketches of \Cref{thm:risk_bound_generalized}, \Cref{thm:conf_set}, and \Cref{cor:lr_conf_set} here and defer their full proofs to \Cref{sec:proofs}.
We discuss in~\Cref{sub:discussion} 
how our proof techniques and theoretical results complement and improve on previous works.

We start by showing the existence and uniqueness of $\theta_n$.
The next result shows that $\theta_n$ exists and is unique whenever the quadratic form $\grad_n(\theta_\star)^\top H_n^{-1}(\theta_\star) \grad_n(\theta_\star)$ is small.
Note that this quantity is also known as Rao's score statistic for goodness-of-fit testing.
This result also localizes $\theta_n$ to a neighborhood of the target parameter $\theta_\star$.
\begin{proposition}\label{prop:localization}
    Under \Cref{asmp:self_concordance},
    if $\norm{\grad_n(\theta_\star)}_{H_n^{-1}(\theta_\star)} \le C_{\nu} [\lambda_{\min}(H_n(\theta_\star))]^{(3-\nu)/2} / (R n^{\nu/2-1})$,
    then the estimator $\theta_n$ uniquely exists and satisfies
    \begin{align*}
        \norm{\theta_n - \theta_\star}_{H_n(\theta_\star)} \le 4 \norm{\grad_n(\theta_\star)}_{H_n^{-1}(\theta_\star)}.
    \end{align*}
\end{proposition}

The main tool used in the proof of \Cref{prop:localization} is a strong convexity type result for generalized self-concordant functions recalled in \Cref{sec:tools}.
In order to apply \Cref{prop:localization}, we need to control $\norm{\grad_n(\theta_\star)}_{H_n^{-1}(\theta_\star)}$.
This result is summarized in the following proposition.

\begin{proposition}\label{prop:score}
    Under \Cref{asmp:sub_gaussian,asmp:bernstein} with $r = 0$, it holds that, with probability at least $1 - \delta$,
    \begin{align*}
        \norm{S_n(\theta_\star)}_{H_n^{-1}(\theta_\star)}^2 \lesssim \frac{d_\star}n \log{(e/\delta)}
    \end{align*}
    whenever $n \gtrsim \log{(2d/\delta)}$.
\end{proposition}

The proof of \Cref{prop:score} consists of two steps: (a) lower bound $H_n(\theta_\star)$ by $H_\star$ up to a constant using the Bernstein inequality and (b) upper bound $\norm{\grad_n(\theta_\star)}_{H^{-1}(\theta_\star)}$ using a concentration inequality for isotropic random vectors, where the tools are recalled in \Cref{sec:tools}.
Combining them implies that $\norm{\grad_n(\theta_\star)}_{H^{-1}(\theta_\star)}$ can be arbitrarily small and thus satisfies the requirement in \Cref{prop:localization} for sufficiently large $n$.
This not only proves the existence and uniqueness of the empirical risk minimizer $\theta_n$ but also provides an upper bound for $\norm{\theta_n - \theta_\star}_{H_n(\theta_\star)}$ through $\norm{\grad_n(\theta_\star)}_{H_n^{-1}(\theta_\star)}$.

In order to prove \Cref{thm:conf_set}, it remains to upper bound $H_n(\theta_n)$ by $H_\star$ up to a constant factor.
This can be achieved by the following result.
\begin{proposition}\label{prop:emp_hess_est}
    Under \Cref{asmp:self_concordance,asmp:bernstein} with $r = C_\nu \lambda_\star^{(\nu-3)/2} / R$, it holds that, with probability at least $1 - \delta$,
    \begin{align*}
        \frac1{2C_\nu} H_\star \preceq H_n(\theta) \preceq \frac32 C_\nu H_\star, \;\mbox{for all } \theta \in \Theta_{r}(\theta_\star),
    \end{align*}
    whenever $n \gtrsim \left\{ \log{(2d/\delta)} + d (\nu/2-1) \log{n}\right\}$.
\end{proposition}

Finally, \Cref{cor:lr_conf_set} follows from \Cref{thm:conf_set} and the Taylor expansion: there exists $\bar \theta_n \in \mbox{Conv}\{\theta_n, \theta_\star\}$ such that
\begin{align*}
    2[L_n(\theta_\star) - L_n(\theta_n)] = \norm{\theta_n - \theta_\star}_{H_n(\bar \theta_n)},
\end{align*}
where we have used $\nabla L_n(\theta_n) = 0$.

\subsection{Approximating the effective dimension}

One downside of \Cref{thm:conf_set,cor:lr_conf_set} is that $d_\star$ depends on the unknown data distribution.
Alternatively, we use the following empirical counterpart
\begin{align*}
    d_n := \Tr\left(H_n(\theta_n)^{-1/2} G_n(\theta_n) H_n(\theta_n)^{-1/2} \right).
\end{align*}
The next result implies that we do not lose much if we replace $d_\star$ by $d_n$.
This result is novel and of independent interest since one also needs to estimate $d_\star$ in order to construct asymptotic confidence sets under model misspecification.

\begin{customasmp}{2'}\label{asmp:subG_local}
    There exist constants $r, K_1 > 0$ such that, for any $\theta \in \Theta_r(\theta_\star)$, we have $\norm{G(\theta)^{-1/2} S(\theta; Z)}_{\psi_2} \le K_1$.
\end{customasmp}

\begin{assumption}\label{asmp:lip}
    There exists $r > 0$ such that $M := \Expect[M(Z)] < \infty$, where $M(z)$ is defined as
    \begin{align*}
        \sup_{\theta_1 \neq \theta_2 \in \Theta_r(\theta_\star)} \frac{\norm{G_\star^{-1/2} [G(\theta_1; z) - G(\theta_2; z)] G_\star^{-1/2}}_2}{\norm{\theta_1 - \theta_2}_{H_\star}}.
    \end{align*}
\end{assumption}

\myparagraph{Remark}
\Cref{asmp:lip} is a Lipschitz-type condition for $G(\theta; z)$. This assumption was previously used by \citep[Assumption 3]{mei2018landscape} to analyze non-convex risk landscapes. 

\begin{proposition}\label{prop:d_n}
    Let $\nu \in [2, 3)$.
    Under Asms.~\ref{asmp:self_concordance}, \ref{asmp:subG_local}, \ref{asmp:bernstein} and \ref{asmp:lip} with $r = C_\nu \lambda_\star^{(3-\nu)/2}/R$, it holds that
    \begin{align*}
        \frac1{C_\nu} d_\star \le d_n \le C_\nu d_\star,
    \end{align*}
    with probability at least $1 - \delta$,
    whenever $n$ is large enough (see \Cref{sub:appendix:consist_dn} for the precise condition).
\end{proposition}

\myparagraph{Remark}
The precise version of $\Cref{prop:d_n}$ in \Cref{sub:appendix:consist_dn} implies that $d_n$ is a consistent estimator of $d$.

With \Cref{prop:d_n} at hand, we can obtain finite-sample confidence sets involving $d_n$, which can be computed from data.
We illustrate it with the Wald confidence set.
\begin{corollary}\label{cor:wald_conf_set}
    Suppose the same assumptions in \Cref{prop:d_n} hold true.
    Let $\calC_{\text{Wald}, n}'(\delta)$ be
    \begin{align*}
        \left\{ \theta \in \Theta: \norm{\theta - \theta_\star}_{H_n(\theta_n)}^2 \le C_{K_1, \nu} \log{(e/\delta)} \frac{d_n}{n} \right\}.
    \end{align*}
    Then we have $\Prob(\theta_\star \in \calC_{\text{Wald}, n}'(\delta)) \ge 1 - \delta$ whenever $n$ satisfies the same condition as in \Cref{prop:d_n}.
\end{corollary}

\subsection{Discussion}
\label{sub:discussion}

\myparagraph{Fisher information and model misspecification}
When the model is well-specified, the autocorrelation matrix $G(\theta)$ coincides with the well-known Fisher information $\mathcal{I}(\theta) := \Expect_{Z \sim P_\theta}[S(\theta; Z)S(\theta; Z)^\top]$ at $\theta_\star$.
The Fisher information plays a central role in mathematical statistics and, in particular, M-estimation; see \citep{pennington2018spectrum,kunstner2019limitations,ash2021gone,soen2021variance} for recent developments in this line of research.
It quantifies the amount of information a random variable carries about the model parameter.
Under a well-specified model, it also coincides with the Hessian matrix $H(\theta)$ at the optimum which captures the local curvature of the population risk.
When the model is misspecified, the Fisher information deviates from the Hessian matrix.
In the asymptotic regime, this discrepancy is reflected in the limiting covariance of the weighted M-estimator which admits a sandwich form $H_\star^{-1/2} G_\star H_\star^{-1/2}$; see, e.g., \cite[Sec.~4]{huber1967under}.

\myparagraph{Effective dimension}
The counterpart of the sandwich covariance in the non-asymptotic regime is the effective dimension $d_\star$; see, e.g., \citep{spokoiny2017penalized,ostrovskii2021finite}.
Our bounds also enjoy the same merit---its dimension dependency is via the effective dimension.
When the model is well-specified, the effective dimension reduces to $d$, recovering the same rate of convergence $O(d/n)$ as in classical linear regression; see, e.g., \cite[Prop.~3.5]{bach2021learning}.
When the model is misspecified, the effective dimension provides a characterization of the problem complexity which is adapted to both the data distribution and the loss function via the matrix $H_\star^{-1/2} G_\star H_\star^{-1/2}$.
To gain a better understanding of the effective dimension $d_\star$, we summarize it in \Cref{tab:decay} in \Cref{sec:proofs} under different regimes of eigendecay, assuming that $G_\star$ and $H_\star$ share the same eigenvectors.
It is clear that, when the spectrum of $G_\star$ decays faster than the one of $H_\star$, the dimension dependency can be better than $O(d)$.
In fact, it can be as good as $O(1)$ when the spectrum of $G_\star$ and $H_\star$ decay exponentially and polynomially, respectively.

\myparagraph{Comparison to classical asymptotic theory}
Classical asymptotic theory of M-estimation is usually based on two assumptions: (a) the model is well-specified and (b) the sample size $n$ is much larger than the parameter dimension $d$.
These assumptions prevent it from being applicable to many real applications where the parametric family is only an approximation to the unknown data distribution and the data is of high dimension involving a large number of parameters.
On the contrary, our results do not require a well-specified model, and the dimension dependency is replaced by the effective dimension $d_\star$ which captures the complexity of the parameter space.
Moreover, they are of non-asymptotic nature---they hold true for any $n$ as long as it exceeds some constant factor of $d_\star$.
This allows the number of parameters to potentially grow with the same size.

\myparagraph{Comparison to recent non-asymptotic theory}
Recently, \citet{spokoiny2012parametric} achieved a breakthrough in finite-sample analysis of parametric M-estimation.
Although fully general, their results require strong global assumptions on the deviation of the empirical risk process and are built upon advanced tools from empirical process theory.
Restricting ourselves to generalized self-concordant losses, we are able to provide a more transparent analysis with neater assumptions only in a neighborhood of the optimum parameter $\theta_\star$.
Moreover, our results maintain some generality, covering several interesting examples in statistical machine learning as provided in \Cref{sub:examples}.

\citet{ostrovskii2021finite} also considered self-concordant losses for M-estimation.
However, their results are limited to generalized linear models whose loss is (pseudo) self-concordant and admits the form $\ell(\theta; Z) := \ell(Y, \theta^\top X)$.
While sharing the same rate $O(d_\star / n)$, our results are more general than theirs in two aspects.
First, the loss need not be of the form $\ell(Y, \theta^\top X)$, encompassing the score matching loss in \Cref{ex:score_matching} below.
Second, we go beyond pseudo self-concordance via the notion of generalized self-concordance.
Moreover, they focus on bounding the excess risk rather than providing confidence sets, and they do not study the estimation of $d_\star$.

Pseudo self-concordant losses have been considered for semi-parametric models \citep{liu2022orthogonal}.
However, they focus on bounding excess risk and require a localization assumption on $\theta_n$. Here we prove the localization result in \Cref{prop:localization} and we focus on confidence sets.

\myparagraph{Regularization}
Our results can also be applied to regularized empirical risk minimization by including the regularization term in the loss function.
Let $\theta_{n}^\lambda$ and $\theta_{\star}^\lambda$ be the minimizers of the \emph{regularized} empirical and population risk, respectively.
Let $d_\star^\lambda := \Tr\big((H_\star^\lambda)^{-1/2} G_\star^{\lambda} (H_\star^\lambda)^{-1/2}\big)$ where $H_\star^{\lambda}$ and $G_\star^{\lambda}$ are the regularized Hessian and the autocorrelation matrix of the regularized gradient at $\theta_\star^\lambda$, respectively.
Then our results characterize the concentration of $\theta_{n}^\lambda$ around $\theta_{\star}^\lambda$:
\begin{align*}
    \norm{\theta_n^{\lambda} - \theta_\star^\lambda}_{H_\star^\lambda}^2 \le O(d_\star^\lambda / n).
\end{align*}
This result coincides with \citet[Thm.~2.1]{spokoiny2017penalized}.
If the goal is to estimate the unregularized population risk minimizer $\theta_\star$, then we need to pay an additional error $\norm{\theta_\star^\lambda - \theta_\star}_{H_\star^\lambda}^2$ which is referred to as the modeling bias \citep[Sec.~2.5]{spokoiny2017penalized}.
One can invoke a so-called \emph{source condition} to bound the modeling bias and a \emph{capacity condition} to bound $d_\star^\lambda$.
An optimal value of $\lambda$ can be obtained by balancing between these two terms \cite[see, e.g.,][]{marteau2019beyond}.

For instance, let $Z := (X, Y)$ where $X \in \reals^d$ with $\Expect[XX^\top] = I_d$ and $Y \in \reals$.
Consider the regularized squared loss
$ \score^\lambda(\theta; z) := 1/2\, (y - \theta^\top x)^2 + 1/2\, \theta^\top U \theta$
where $U = \diag\{\mu_1, \dots, \mu_d\}$.
The regularized effective dimension is then~\citep[Sec.~2.1]{spokoiny2017penalized} of order 
$ O\big( \sum_{k=1}^d 1/(1 + \mu_k) \big)$
which can be much smaller than $d$ if $\{\mu_k\}$ is increasing.

\section{EXAMPLES AND APPLICATIONS}
\label{sec:application}
We give several examples whose loss function is generalized self-concordant so that our results can be applied.
We also provide finite-sample analysis for Rao's score test, the likelihood ratio test, and the Wald test in goodness-of-fit testing.
All the proofs and derivations are deferred to \Cref{sec:example}.

\subsection{Examples}
\label{sub:examples}

\begin{example}[Generalized linear models]\label{ex:glm}
    Let $Z := (X, Y)$ be a pair of input and output, where $X \in \calX \subset \reals^d$ and $Y \in \calY \subset \reals$.
    Let $t: \calX \times \calY \rightarrow \reals^d$ and $\mu$ be a measure on $\calY$.
    Consider the statistical model
    \begin{align*}
        p_\theta(y \mid x) \sim \frac{\exp(\theta^\top t(x, y))}{\int \exp(\theta^\top t(x, \bar y)) \D \mu(\bar y)} \D \mu(y)
    \end{align*}
    with $\norm{t(X, Y)}_2 \le_{a.s.} M$.
    It induces the loss function
    \begin{align*}
        \score(\theta; z) := -\theta^\top t(x, y) + \log{\int \exp(\theta^\top t(x, \bar y)) \D \mu(\bar y)},
    \end{align*}
    which is generalized self-concordant for $\nu = 2$ and $R = 2M$.
    Moreover, this model satisfies \Cref{asmp:sub_gaussian,asmp:bernstein,asmp:lip} and \ref{asmp:subG_local}.
\end{example}

\begin{example}[Score matching with exponential families]\label{ex:score_matching}
    Assume that $\bbZ = \reals^p$.
    Consider an exponential family on $\reals^d$ with densities
    \begin{align*}
        \log{p_\theta(z)} = \theta^\top t(z) + h(z) - \Lambda(\theta).
    \end{align*}
    The non-normalized density $q_\theta$ then reads $\log{q_\theta(z)} = \theta^\top t(z) + h(z)$.
    As a result, the score matching loss becomes
    \begin{align*}
        \score(\theta; z) = \frac12 \theta^\top A(z) \theta - b(z)^\top \theta + c(z) + \text{const},
    \end{align*}
    where $A(z) := \sum_{k=1}^p \frac{\partial t(z)}{\partial z_k} \big(\frac{\partial t(z)}{\partial z_k}\big)^\top$ is positive semi-definite, $b(z) := \sum_{k=1}^p \left[ \frac{\partial^2 t(z)}{\partial z_k^2} + \frac{\partial h(z)}{\partial z_k} \frac{\partial t(z)}{\partial z_k} \right]$, and $c(z) := \sum_{k=1}^p \left[ \frac{\partial^2 h(z)}{\partial z_k^2} + \big(\frac{\partial h(z)}{\partial z_k}\big)^2 \right]$.
    Therefore, the score matching loss $\score(\theta; z)$ is convex.
    Moreover, since the third derivatives of $\ell(\cdot; z)$ is zero, the score matching loss is generalized self-concordant for all $\nu \ge 2$ and $R \ge 0$.
\end{example}

\subsection{Rao's Score Test and Its Relatives}
\label{sub:goodness}

We discuss how our results can be applied to analyze three classical goodness-of-fit tests.
In this subsection, we will assume that the model is well-specified.
Due to \Cref{asmp:proper_loss}, we will use $\theta_\star$ to denote the true parameter of $\Prob$ and reserve $\theta_0$ for the parameter under the null hypothesis.

Given a subset $\Theta_0 \subset \Theta$, a goodness-of-fit testing problem is to test the hypotheses
\begin{align*}
    \hnull: \theta_\star \in \Theta_0 \leftrightarrow \halt: \theta_\star \notin \Theta_0.
\end{align*}
We focus on a simple null hypothesis where $\Theta_0 := \{\theta_0\}$ is a singleton.
A statistical test consists of a test statistic $T := T(Z_1, \dots, Z_n)$ and a prescribed critical value $t$, and we reject the null hypothesis if $T > t$.
Its performance is quantified by the \emph{type I error rate} $\Prob(T > t \mid \hnull)$ and \emph{statistical power} $\Prob(T > t \mid \halt)$.
Classical goodness-of-fit tests include Rao's score test, the likelihood ratio test (LRT), and the Wald test.
Their test statistics are $\rao := \norm{\grad_n(\theta_0)}_{H_n^{-1}(\theta_0)}^2$, $\lr := 2[\score_n(\theta_0) - \score_n(\theta_n)]$, and $\wald := \norm{\theta_n - \theta_0}_{H_n(\theta_n)}^2$,
respectively.

Our approach can be applied to analyze the type I error rate of these tests as summarized in the following proposition.
\begin{proposition}[Type I error rate]\label{prop:typeI}
    Suppose that \Cref{asmp:sub_gaussian,asmp:bernstein} with $r = 0$ hold true.
    Under $\hnull$, we have, with probability at least $1 - \delta$,
    \begin{align*}
        \rao \lesssim \log{(e/\delta)} \frac{d}n
    \end{align*}
    whenever $n \gtrsim \log{(2d/\delta)}$.
    Furthermore, if \Cref{asmp:self_concordance,asmp:sub_gaussian,asmp:bernstein} with $r = C_\nu \lambda_\star^{(\nu-3)/2}/R$ hold true, we have, with probability at least $1 - \delta$,
    \begin{align*}
        \lr, \wald \lesssim \log{(e/\delta)} \frac{d}{n}
    \end{align*}
    whenever $n$ satisfies \eqref{eq:n_large_enough}.
\end{proposition}

This result implies that the three test statistics all scale as $O(d / n)$ under the null hypothesis.
Consequently, for a fixed significance level $\alpha \in (0, 1)$, we can choose the critical value $t = t_n(\alpha) = O(d/n)$ so that their type I error rates are below $\alpha$.
With this choice, we can then characterize the statistical powers of these tests under alternative hypotheses $\theta_\star \neq \theta_0$ where $\theta_\star$ may depend on $n$.
Let $\Omega(\theta) := G(\theta)^{1/2} H(\theta)^{-1} G(\theta)^{1/2}$ and $h(\tau) := \min\{\tau^2, \tau\}$.

\begin{proposition}[Statistical power]
\label{prop:power}
    Let $\theta_\star \neq \theta_0$.
    The following statements are true for sufficiently large $n$.
    \begin{enumerate}
        \item[(a)] Suppose that \Cref{asmp:self_concordance,asmp:sub_gaussian,asmp:bernstein} hold true with $r=0$.
        When $\theta_\star - \theta_0 = O(n^{-1/2})$ and $\tau_n := t_n(\alpha)/4 - \norm{S(\theta_0)}_{H(\theta_0)^{-1}}^2 - \Tr(\Omega(\theta_0))/n > 0$, we have
        \begin{align*}
            &\quad \Prob(\rao > t_n(\alpha)) \\
            &\le 2d e^{- C_{K_2, \sigma_H} n} + e^{-C_{K_1} h(n \tau_n/\norm{\Omega(\theta_0)}_2)}.
        \end{align*}
        When $\theta_* - \theta_n = \omega(n^{-1/2})$, we have
        \begin{align*}
            &\quad \Prob(\rao > t_n(\alpha)) \\
            &\ge 1 - 2d e^{-C_{K_2, \sigma_H} n} - e^{-C_{K_1} n \bar \tau_n/\norm{\Omega(\theta_0)}_2},
        \end{align*}
        where $\bar \tau_n = \Theta(\norm{\theta_\star - \theta_n}^2)$.
        
        \item[(b)] Suppose that the assumptions in \Cref{thm:conf_set} hold true.
        When $\theta_\star - \theta_0 = O(n^{-1/2})$ and $\tau_n' := t_n(\alpha)/384 - \norm{\theta_\star - \theta_0}_{H(\theta_\star)}^2/64 - d/n > 0$, we have
        \begin{align*}
            &\quad \Prob(\lr > t_n(\alpha)) \\
            &\le e^{-C_{K_1} h(n\tau_n'/\norm{\Omega(\theta_\star)}_2)} + e^{-C_{K_1, \nu} (\lambda_\star n)^{3-\nu}/(R^2 d)}.
        \end{align*}
        When $\theta_* - \theta_n = \omega(n^{-1/2})$, we have
        \begin{align*}
            &\quad \Prob(\lr > t_n(\alpha)) \\
            &\ge 1 - e^{-C_{K_1} \frac{n \bar \tau_n'}{\norm{\Omega(\theta_\star)}_2}} - e^{-\frac{C_{K_1, \nu} (\lambda_\star n)^{3-\nu}}{R^2 d}},
        \end{align*}
        where $\bar \tau_n' = \Theta(\norm{\theta_\star - \theta_n}^2)$.
        
        \item[(c)] The same statements replacing $\lr$ by $\wald$.
    \end{enumerate}
\end{proposition}

According to \Cref{prop:power}, when $\theta_\star - \theta_0 = O(n^{-1/2})$, the powers of the three tests are asymptotically upper bounded; when $\theta_\star - \theta_0 = \omega(n^{-1/2})$, the power of Rao's score test tends to one at rate $O(e^{-n \norm{\theta_\star - \theta_0}^2})$ and the ones of the other two tests tend to one at rate $O(e^{-n \norm{\theta_\star - \theta_0}^2 \wedge n^{3-\nu}})$.

\section{NUMERICAL STUDIES}
\label{sec:experiments}
\begin{figure}
  \centering
  \includegraphics[width=0.6\textwidth]{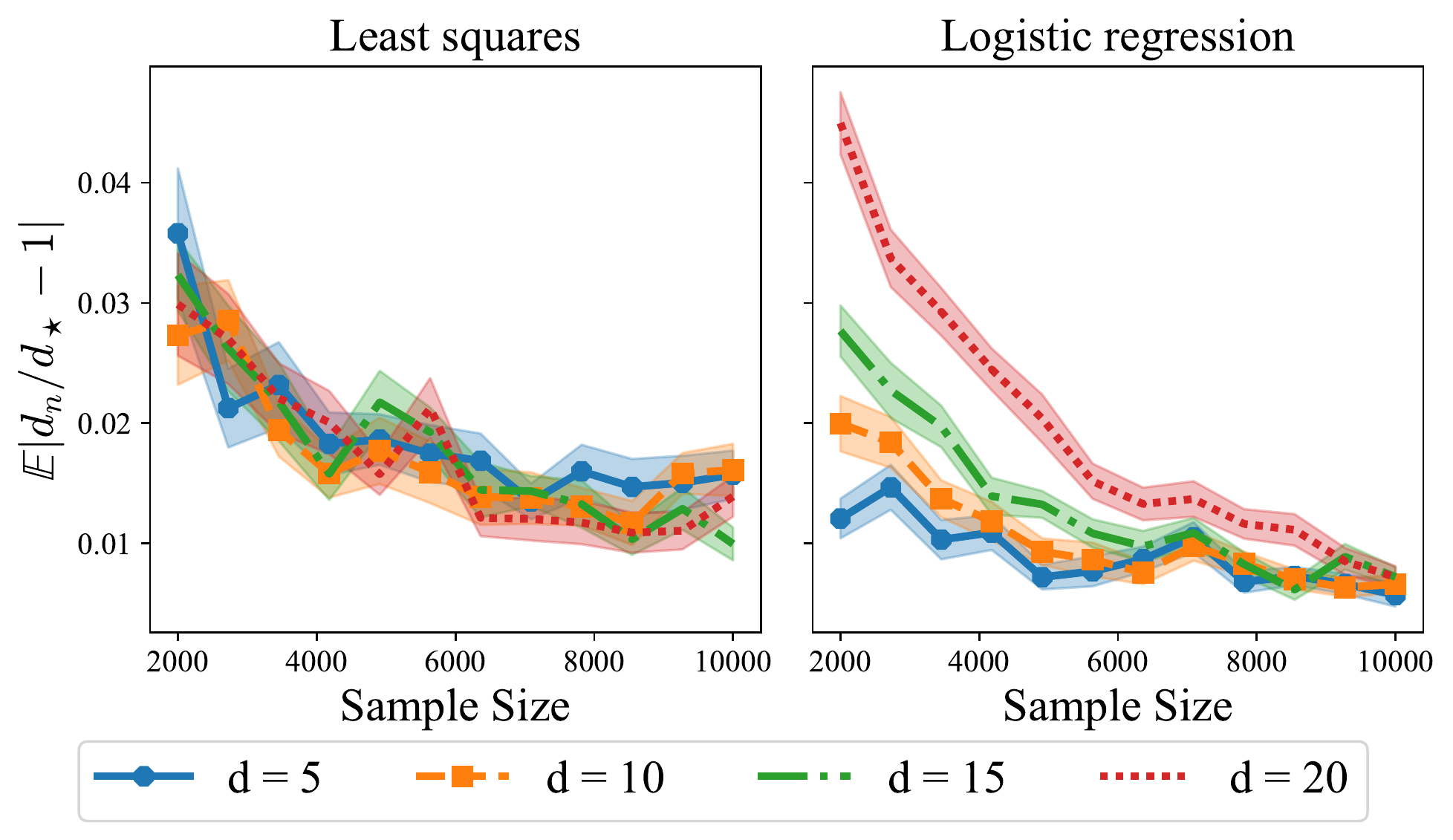}
  \caption{Absolute error of the empirical effective dimension. \textbf{(Left)}: least squares; \textbf{(Right)}: logistic regression.}
  \label{fig:est_effective_dim}
\end{figure}

We run simulation studies to illustrate our theoretical results.
We start by demonstrating the consistency of $d_n$ and the shape of the Wald confidence set defined in \Cref{cor:wald_conf_set}, i.e.,
\begin{align*}
  \calC_{\text{Wald}, n}'(\delta) = \left\{\theta \in \Theta: \norm{\theta - \theta_n}_{H_n(\theta_n)}^2 \le C_{K_1,\nu} \frac{d_n}{n} \log{(e/\delta)} \right\}.
\end{align*}
Note that the oracle Wald confidence set should be constructed from $\norm{\theta_n - \theta_\star}_{H_\star}$ and $d_\star$; however, \Cref{cor:wald_conf_set} suggests that we can replace $H_\star$ and $d_\star$ by $H_n(\theta_n)$ and $d_n$ without losing too much.
To empirically verify our theoretical results, we calibrate the Wald confidence set based on $\norm{\theta_n - \theta_\star}_{H_n(\theta_n)}$ with the threshold from the oracle Wald confidence set and compare its coverage with the one calibrated by the multiplier bootstrap---a popular resampling-based approach for calibration.
Finally, we compare the coverage of the Wald and LR confidence sets calibrated by the multiplier bootstrap.
In all the experiments, we generate $n$ i.i.d.~pairs by sampling $X$ and then sampling $Y \mid X$.

\subsection{Numerical Illustrations}

\paragraph{Approximation of the effective dimension.}
By \Cref{prop:d_n}, we know that $d_n$ is a consistent estimator of $d_\star$.
We verify it with simulations.
We consider two models.
For least squares, the data are generated from $X \sim \calN(0, I_d)$ and $Y | X \sim \calN(\ones^\top X, 1)$.
For logistic regression, the data are generated from $X \sim \calN(0, I_d)$ and $Y \mid X \sim p(Y \mid X) = \sigma(Y \ones^\top X)$ for $Y \in \{-1, 1\}$ where $\sigma(u) := (1 + e^{-u})^{-1}$.
We then estimate $d_\star = d$ (since the model is well-specified) by $d_n$ and quantify its estimation error by $\Expect\abs{d_n / d_\star - 1}$.
We vary $n \in [2000, 10000]$ and $d \in \{5, 10, 15, 20\}$, and give the plots in \Cref{fig:est_effective_dim}.
For a fixed $d$, the absolute error decays to zero as the sample size increases as predicted by \Cref{prop:d_n}.
For a fixed $n$, the absolute error raises as the dimension becomes larger in logistic regression, but it remains similar in least squares.

\begin{figure}
  \centering
  \includegraphics[width=0.7\textwidth]{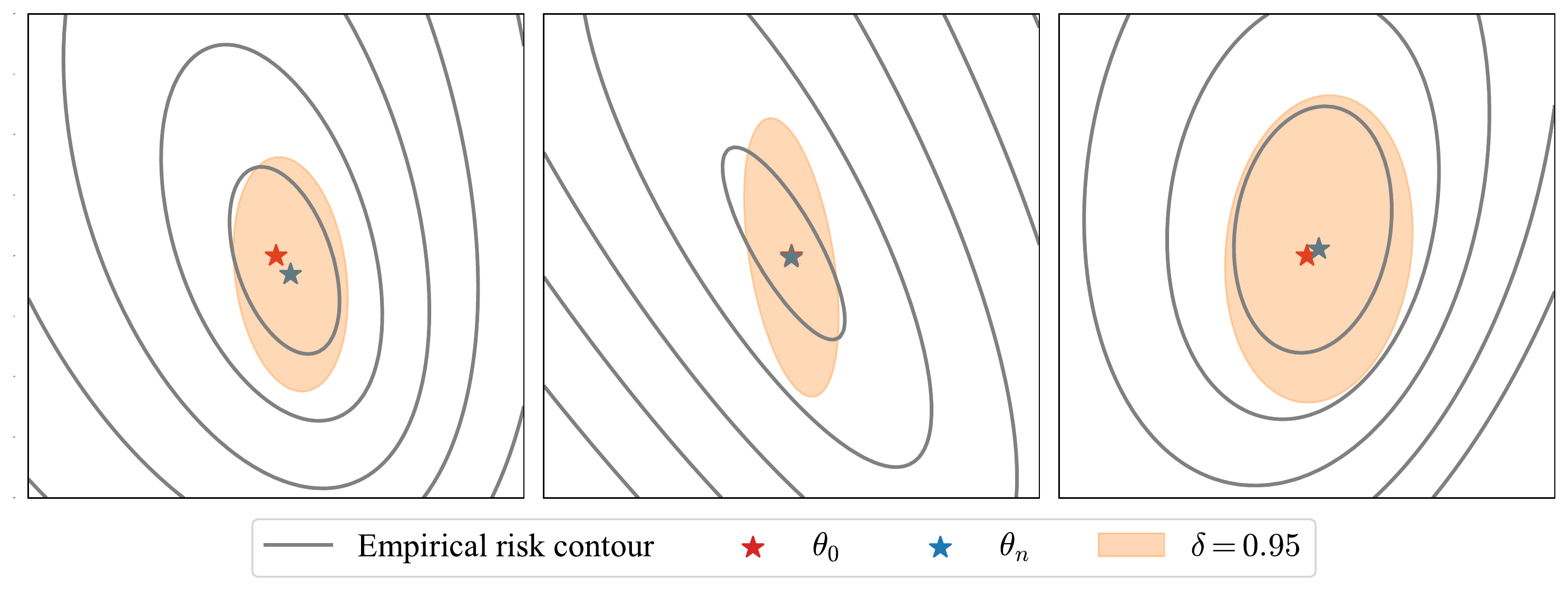}
  \caption{Confidence set in \Cref{cor:wald_conf_set} under a logistic regression model. \textbf{Left:} $\Sigma = (2, 0; 0, 1)$; \textbf{Middle:} $\Sigma = (2, 1; 1, 1)$; \textbf{Right:} $\Sigma = (2, -1; -1, 1)$.}
  \label{fig:logistic_conf_set}
\end{figure}

\paragraph{Shape of the Wald confidence set.}
Recall that the Wald confidence set in \Cref{thm:conf_set} is an ellipsoid whose shape is determined by the empirical Hessian $H_n(\theta_n)$ and thus can effectively handles the local curvature of the empirical risk.
We illustrate this feature on a logistic regression example.
We generate data from $X \sim \calN(0, \Sigma)$ with different $\Sigma$'s and $Y \mid X \sim p(Y \mid X) = \sigma(Y \theta_0^\top X)$ for $Y \in \{-1, 1\}$ where $\theta_0 = (-1, 2)^\top$.
We then construct the confidence set with $d_\star = d$.
As shown in \Cref{fig:logistic_conf_set}, the shape of the confidence set varies with $\Sigma$ and captures the curvature of the empirical risk at $\theta_0$.

\subsection{Calibration}
\label{sub:bootstrap}

We investigate two calibration schemes.
Inspired by the setting in \citet[Sec.~5.1]{chen2020robust},
we generate $n = 100$ i.i.d.~observations from three models with true parameter $\theta_0$ whose elements are equally spaced between $[0, 1]$---1) \emph{well-specified least squares} with $X \sim \calN(0, I_d)$ and $Y \mid X \sim \calN(\theta_0^\top X, 1)$, 2) \emph{misspecified least squares} with $X \sim \calN(0, I_d)$ and $Y \mid X \sim \theta_0^\top X + t_{3.5}$, and 3) \emph{well-specified logistic regression} with $X \sim \calN(0, I_d)$ and $Y \mid X \sim p(Y \mid X) = \sigma(Y \theta_0^\top X)$ for $Y \in \{-1, 1\}$.
For each $\delta \in \{0.95, 0.9, 0.85, 0.8, 0.75\}$, we construct a confidence set using either \emph{oracle calibration} or \emph{multiplier bootstrap}.
We repeat the whole process for $1000$ times and report the coverage of each confidence set in \Cref{tab:bootstrap}.

\begin{table*}[t]
  \caption{Coverage of the oracle and bootstrap confidence sets.}
  \label{tab:bootstrap}
  \centering
  \renewcommand{\arraystretch}{1.2}
  \begin{tabular}{llccccc}
      \addlinespace[0.4em]
      \toprule
      \multicolumn{1}{c}{\textbf{Model}} & \multicolumn{1}{c}{\textbf{Confidence set}} & \multicolumn{1}{c}{$\delta = 0.95$} & \multicolumn{1}{c}{$\delta = 0.9$} & \multicolumn{1}{c}{$\delta = 0.85$} & \multicolumn{1}{c}{$\delta = 0.8$} & \multicolumn{1}{c}{$\delta = 0.75$}  \\
      \midrule
      \multirow{3}{*}{Well-specified least squares} & Oracle & 0.957 & 0.908 & 0.868 & 0.792 & 0.770 \\
      & BootWald & 0.947 & 0.908 & 0.855 & 0.791 & 0.735 \\
      & BootLR & 0.949 & 0.906 & 0.852 & 0.792 & 0.737 \\
      \midrule
      \multirow{3}{*}{Misspecified least squares} & Oracle & 0.972 & 0.916 & 0.882 & 0.841 & 0.764 \\
      & BootWald & 0.968 & 0.924 & 0.865 & 0.779 & 0.727 \\
      & BootLR & 0.972 & 0.923 & 0.865 & 0.784 & 0.727 \\
      \midrule
      \multirow{3}{*}{Well-specified logistic regression} & Oracle & 0.961 & 0.915 & 0.868 & 0.809 & 0.776 \\
      & BootWald & 0.938 & 0.885 & 0.826 & 0.781 & 0.706 \\
      & BootLR & 0.976 & 0.948 & 0.901 & 0.866 & 0.791 \\
      \bottomrule
  \end{tabular}
\end{table*}

\paragraph{Oracle calibration.}
According to \Cref{thm:risk_bound_generalized}, if we have access to $H_\star$ and $d_\star$, we can construct a confidence set of the form $\calC_\star(\delta) := \{\theta: \norm{\theta_n - \theta}_{H_\star} \le d_\star/n + c_n(\delta)\}$.
Now \Cref{cor:wald_conf_set} suggests that $H_\star$ and $d_\star$ can be accurately estimated by $H_n(\theta_n)$ and $d_n$, respectively, leading the confidence set $\calC_n(\delta) := \{\theta: \norm{\theta_n - \theta}_{H_n(\theta_n)} \le d_n/n + c_n(\delta)\}$.
To calibrate $\calC_n(\delta)$, we use the data generating distribution to estimate $c_n(\delta)$ so that $\Prob(\theta_\star \in \calC_\star(\delta)) \approx 1 - \delta$, and then plug it into $\calC_n(\delta)$.
We call it the \emph{oracle Wald confidence set}.
As shown in \Cref{tab:bootstrap}, its coverage is very close to the prescribed confidence level in the well-specified case and it tends to be more conservative in the misspecified case.

\paragraph{Multiplier bootstrap.}
To further evaluate the oracle calibration, we compare its coverage with the one calibrated by the multiplier bootstrap \citep[e.g.,][]{chen2020robust}---a popular resampling-based calibration approach that is widely used in practice.
We construct a \emph{bootstrap Wald confidence set} (BootWald) with $B = 2000$ bootstrap samples in the following steps.
For each $b \in \{1, \dots, B\}$, we 1) generate weights $\{W_i^b\}_{i=1}^n \txtover{i.i.d.}{\sim} \calN(1, 1)$, 2) compute the bootstrap estimator
\begin{align*}
  \theta_n^b = \argmin_{\theta} \left[ L_n^b(\theta) := \frac1n \sum_{i=1}^n W_i^b \ell(\theta; Z_i) \right],
\end{align*}
3) compute the bootstrap Wald statistic $T_{\text{Wald}}^b := \norm{\theta_n^b - \theta_n}_{H_n^b(\theta_n^b)}^2$ where $H_n^b(\theta) := \nabla_\theta^2 L_n^b(\theta)$.
Finally, we compare $\norm{\theta_n - \theta_0}_{H_n(\theta_n)}^2$ with the upper $\delta$ quantile of $\{T_{\text{Wald}}^b\}_{b=1}^B$ to decide if the Wald confidence set covers the true parameter.
It is clear that the bootstrap Wald confidence set performs similarly as the oracle Wald confidence set in least squares, but it is more liberal in logistic regression.

For comparison purposes, we also describe the procedure to construct a \emph{bootstrap likelihood ratio confidence set} (BootLR).
The first two steps are the same as the bootstrap Wald confidence set, while the third step is to compute the bootstrap LR statistic $T_{\text{LR}}^b := 2[L_n^b(\theta_n) - L_n^b(\theta_n^b)]$.
And we compare $2[L_n(\theta_0) - L_n(\theta_n)]$ with the upper $\delta$ quantile of $\{T_{\text{LR}}^b\}_{b=1}^B$ to decide if the bootstrap LR confidence set covers the true parameter.
For the well-specified least squares, the two bootstrap confidence sets perform similarly with coverages close to the target ones.
However, when the target coverage is small (i.e., $0.75$), they tend to be liberal.
For the misspecified least squares, the bootstrap two confidence sets perform similarly.
When the target coverage is large, they tend to be conservative; when the target coverage is small, they tend to be liberal.
For the well-specified logistic regression, the bootstrap Wald confidence set tends to be liberal and the bootstrap LR one tends to be conservative.

\subsubsection*{Acknowledgements}
The authors would like to thank K.~Jamieson, L.~Jain, and V. Roulet for fruitful discussions.
L.~Liu is supported by NSF CCF-2019844 and NSF DMS-2023166 and NSF DMS-2133244.
Z.~Harchaoui is supported by NSF CCF-2019844, NSF DMS-2134012, NSF DMS-2023166, CIFAR-LMB, and faculty research awards.
Part of this work was done while Z.~Harchaoui was visiting the Simons Institute for the Theory of Computing.

\clearpage

\bibliographystyle{abbrvnat}
\bibliography{biblio}

\clearpage
\appendix

% Make the appendix single col
\begingroup
\let\clearpage\relax 
\onecolumn 
\endgroup
% End: appendix single col

\addcontentsline{toc}{section}{Appendix} 
\part{Appendix} 
\parttoc
\clearpage

\section{Proof of main results}
\label{sec:proofs}
Our proof techniques rely on a self-concordance property to localize the estimator and control the Hessian and related quantities. This property was, up to our knowledge, first put to use in machine learning by~\citet{abernethy2008efficient} in the context of sequential allocation of experiments and multi-armed bandits. The key observation is that, within the Dikin ellipsoid, the variation of the Hessian
can be easily controlled. More recently,~\citet{ostrovskii2021finite} obtained risk bounds for generalized linear models based on this observation.
Our results and proof techniques also rely on this observation. We show how to leverage this observation to obtain confidence sets for a broad class of statistical models under a generalized self-concordance assumption owing to the use of the matrix Bernstein inequality. For instance, we obtain confidence bounds for parameter estimation using score matching and generalized linear statistical models under possible model misspecification as provided in \Cref{sec:application}.

Our proofs are inspired by \citet{ostrovskii2021finite}.
However, there are two key differences.
First, since they focus on loss functions of the form $\ell(Y, \theta^\top X)$, the Hessian is $\ell''(Y, \theta^\top X) XX^\top$ where $\ell''(y, \bar y) := \D^2 \ell(y, \bar y) / \D \bar y^2$.
As a result, they can control the deviation of the empirical Hessian using inequalities for sample second-moment matrices of sub-Gaussian random vectors \citep[Thm.~A.2]{ostrovskii2021finite}.
In contrast, we use matrix Bernstein inequality which allows us to work with a larger class of loss functions.
Second, we extend their localization result from pseudo self-concordant losses to generalized self-concordant losses (\Cref{prop:localization}).
This is enabled by a new property on the existence of a unique minimizer for generalized self-concordant functions (\Cref{prop:self_concordance_local}). We also establish the concentration of the effective dimension. 

In the remainder of this section, we first prove the localization result \Cref{prop:localization} and the score bound \Cref{prop:score} in \Cref{sub:appendix:local}.
It not only guarantees the existence and uniqueness of $\theta_n$ but also localizes it.
We then, in \Cref{sub:appendix:thm}, control the empirical Hessian at $\theta_n$ as in \Cref{prop:emp_hess_est} using a covering number argument.
Finally, we prove \Cref{thm:risk_bound_generalized}, \Cref{thm:conf_set}, and \Cref{prop:d_n}.

We use the notation $C$ to denote a constant which may change from line to line, where subscripts are used to emphasize the dependency on other quantities.
For instance, $C_d$ represents a quantity depending only on $d$.

\subsection{Localization}
\label{sub:appendix:local}

We start by showing that the empirical risk $L_n$ is generalized self-concordant.
\begin{lemma}\label{lem:emp_risk_self_concordance}
    Under \Cref{asmp:self_concordance}, the empirical risk $L_n$ is $(n^{\nu/2-1} R, \nu)$-generalized self-concordant.
\end{lemma}
\begin{proof}
    By \Cref{asmp:self_concordance}, the loss $\ell(\cdot; Z_i)$ is $(R, \nu)$-generalized self-concordant for every $i \in [n] := \{1, \dots, n\}$.
    Note that $L_n$ is the empirical average of $\{\ell(\cdot; Z_i)\}_{i=1}^n$.
    Hence, it follows from \cite[Prop.~1]{sun2019generalized} that $L_n$ is $(n^{\nu/2-1}R, \nu)$-generalized self-concordant
\end{proof}

Applying \Cref{prop:self_concordance_local} to $L_n$ leads to the localization result.
Let $\lambda_{n,\star} := \lambda_{\min}(H_n(\theta_\star))$ and $\lambda_n^\star := \lambda_{\min}(H_n(\theta_\star))$.
Recall $K_\nu$ from \Cref{cor:K_nu}.
Define
\begin{align}\label{eq:R_n_nu_star}
    R_{n, \nu}^\star :=
    \begin{cases}
        \lambda_{n,\star}^{-1/2} R & \mbox{if } \nu = 2 \\
        (\nu/2 -1) \lambda_{n,\star}^{(\nu - 3)/2} n^{\nu/2-1} R & \mbox{if } \nu \in (2, 3] \\
        (\nu/2 - 1) (\lambda_{n}^\star)^{(\nu - 3)/2} n^{\nu/2-1} R & \mbox{if } \nu > 3.
    \end{cases}
\end{align}

We can then prove \Cref{prop:localization}.
\begin{customprop}{\ref{prop:localization}}
    Under \Cref{asmp:self_concordance},
    whenever $R_{n,\nu}^\star \norm{\grad_n(\theta_\star)}_{H_n^{-1}(\theta_\star)} \le K_\nu$,
    the estimator $\theta_n$ uniquely exists and satisfies
    \begin{align*}
        \norm{\theta_n - \theta_\star}_{H_n(\theta_\star)} \le 4 \norm{\grad_n(\theta_\star)}_{H_n^{-1}(\theta_\star)}.
    \end{align*}
\end{customprop}
\begin{proof}
    The claim follows directly from \Cref{lem:emp_risk_self_concordance,prop:self_concordance_local}.
\end{proof}

\Cref{prop:localization} implies that the empirical risk minimizer uniquely exists if $\norm{\grad_n(\theta_\star)}_{H_n^{-1}(\theta_\star)}$ is small.
Hence, it remains to bound $\norm{\grad_n(\theta_\star)}_{H_n^{-1}(\theta_\star)}$, which can be achieved by controlling $\norm{\grad_n(\theta_\star)}_{H_\star^{-1}}$ and $H_n(\theta_\star)$.
Let $\Omega(\theta) := G(\theta)^{1/2} H(\theta)^{-1} G(\theta)^{1/2}$ and $\Omega_\star := \Omega(\theta_\star)$
Recall from \Cref{def:effective_dim} that $d_\star = \Tr(\Omega_\star)$.
\begin{lemma}\label{lem:score}
    Under \Cref{asmp:sub_gaussian}, it holds that, with probability at least $1 - \delta$,
    \begin{align*}
        \norm{S_n(\theta_\star)}_{H_\star^{-1}}^{2} \le \frac{d_\star}n + C K_{1}^2 \log{(e / \delta)} \frac{\norm{\Omega_\star}_2}{n}.
    \end{align*}
\end{lemma}
\begin{proof}
    By the first order optimality condition, we have $S(\theta_\star) = 0$.
    As a result,
    \begin{align*}
        X := \sqrt{n} G^{-1/2}(\theta_\star) S_n(\theta_\star; Z)
    \end{align*}
    is an isotropic random vector.
    Moreover, it follows from \Cref{lem:sum_subg} that $\norm{X}_{\psi_2} \lesssim K_{1}$.
    Define $J := G^{1/2}_\star H_\star^{-1} G^{1/2}_\star / n$.
    Then we have
    \begin{align*}
        \norm{S_n(\theta_\star)}_{H_\star^{-1}}^{2} = \norm{X}_{J}^2.
    \end{align*}
    Invoking \Cref{thm:isotropic_tail} yields the claim.
\end{proof}

The next result characterizes the concentration of $H_n(\theta_\star)$.
Let
\begin{align}\label{eq:t_n}
    t_n := t_n(\delta) := \frac{2\sigma_H^2}{-K_2 + \sqrt{K_2^2 + 2\sigma_H^2 n/\log{(4d/\delta)}}}.
\end{align}
Note that it decays to 0 at rate $O(n^{-1/2})$ as $n \rightarrow \infty$.
\begin{lemma}\label{lem:hessian}
    Under \Cref{asmp:bernstein} with $r = 0$, it holds that, with probability at least $1 - \delta$,
    \begin{align*}
        (1 - t_n) H_\star \preceq H_n(\theta_\star) \preceq (1 + t_n) H_\star.
    \end{align*}
    Furthermore, if $n \ge 4(K_2 + 2\sigma_H^2) \log{(2d/\delta)}$, we have $t_n \le 1/2$ and thus
    \begin{align*}
        \frac12 H_\star \preceq H_n(\theta_\star) \preceq \frac32 H_\star.
    \end{align*}
\end{lemma}
\begin{proof}
    Due to \Cref{asmp:bernstein}, the standardized Hessian at $\theta_\star$
    \begin{align*}
        H_\star^{-1/2} H(\theta_\star; Z) H_\star^{-1/2} - I_d
    \end{align*}
    satisfies a Bernstein condition with parameter $K_2$.
    It then follows from \Cref{thm:bernstein_matrix} that
    \begin{align*}
        \Prob\left( \anorm{H_\star^{-1/2} H_n(\theta_\star) H_\star^{-1/2} - I_d}_2 \ge t \right) \le 2d \exp\left\{-\frac{nt^2}{2(\sigma_H^2 + K_2 t)} \right\}.
    \end{align*}
    As a result, it holds that, with probability at least $1 - \delta$,
    $(1 - t_n)I_d \preceq H_\star^{-1/2} H_n(\theta_\star) H_\star^{-1/2} \le (1 + t_n)I_d$, or equivalently,
    \begin{align*}
        (1 - t_n) H_\star \preceq H_n(\theta_\star) \preceq (1 + t_n) H_\star.
    \end{align*}
    Hence, whenever $n \ge 4(K_2 + 2\sigma_H^2) \log{(2d/\delta)}$, we have
    \begin{align*}
        \frac12 H_\star \preceq H_n(\theta_\star) \preceq \frac32 H_\star.
    \end{align*}
\end{proof}

We then prove \Cref{prop:score}.
Recall $t_n$ from \eqref{eq:t_n}.
\begin{customprop}{\ref{prop:score}}
    Under \Cref{asmp:sub_gaussian,asmp:bernstein} with $r = 0$, if $n \ge 4(K_2 + 2\sigma_H^2) \log{(4d/\delta)}$, then we have $t_n \le 1/2$ and, with probability at least $1 - \delta$,
    \begin{align*}
        \norm{S_n(\theta_\star)}_{H_n^{-1}(\theta_\star)}^2
        \le \frac{d_\star}{n(1-t_n)} + C K_{1}^2 \log{(e / \delta)} \frac{\norm{\Omega_\star}_2}{n(1-t_n)}.
    \end{align*}
\end{customprop}

\begin{proof}
    Define two events
    \begin{align*}
        \calA := \left\{ \norm{S_n(\theta_\star)}_{H_\star^{-1}}^{2} \le \frac{d_\star}{n} + C K_{1}^2 \log{(2e / \delta)} \frac{\norm{\Omega_\star}_2}{n} \right\} \quad \mbox{and} \quad \calB := \left\{ (1-t_n) H_\star \preceq H_n(\theta_\star) \preceq (1+t_n) H_\star \right\}.
    \end{align*}
    According to \Cref{lem:score,lem:hessian}, we have $\Prob(\calA) \ge 1 - \delta/2$ and $\Prob(\calB) \ge 1 - \delta/2$.
    On the event $\calA \calB$, we have
    \begin{align*}
        \norm{\grad_n(\theta_\star)}_{H_n^{-1}(\theta_\star)}^2
        \le \frac1{1-t_n} \norm{S_n(\theta_\star)}_{H_\star^{-1}}^2
        \le \frac{d_\star}{n(1-t_n)} + C K_{1}^2 \log{(2e / \delta)} \frac{\norm{\Omega_\star}_2}{n(1-t_n)}.
    \end{align*}
    Since $\Prob(\calA \calB) \ge 1 - \Prob(\calA^c) - \Prob(\calB^c) \ge 1 - \delta$, we have, with probability at least $1 - \delta$,
    \begin{align*}
        \norm{S_n(\theta_\star)}_{H_n^{-1}(\theta_\star)}^2 \le \frac{d_\star}{n(1-t_n)} + C K_{1}^2 \log{(e / \delta)} \frac{\norm{\Omega_\star}_2}{n(1-t_n)}.
    \end{align*}
    If $n \ge 4(K_2 + 2\sigma_H^2) \log{(4d/\delta)}$, then $t_n \le 1/2$ and thus
    \begin{align*}
        \norm{\grad_n(\theta_\star)}_{H_n^{-1}(\theta_\star)}^2 \le \frac{2d_\star}{n} + C K_{1}^2 \log{(e / \delta)} \frac{\norm{\Omega_\star}_2}{n}.
    \end{align*}
\end{proof}

\subsection{Proof of the main theorems}
\label{sub:appendix:thm}

Before we prove the main theorem, we control the empirical Hessian as in \Cref{prop:emp_hess_est}.
A na\"ive approach is to invoke \Cref{lem:hessian} to bound $H_n(\theta)$ by $H_n(\theta_\star)$.
However, this would not work since the generalized self-concordance parameter of $L_n$, i.e., $n^{\nu/2-1} R$, is diverging as $n \rightarrow \infty$.
Hence, we use a covering number argument: 1) we take a covering with radius $O(n^{1-\nu/2})$; 2) we bound $H_n(\theta)$ by $H_n(\pi(\theta))$ where $\pi(\theta)$ is the projection of $\theta$ onto the covering. The factor $n^{1-\nu/2}$ in the radius will cancel out with the factor $n^{\nu/2-1}$ in the generalized self-concordance parameter; 3) we bound $H_n(\pi(\theta))$ by $H(\pi(\theta))$ using matrix concentration; 4) we bound $H(\pi(\theta))$ by $H(\theta_\star)$ where the generalized self-concordance parameter of $L$ is $R$.
Recall $t_n$ from \eqref{eq:t_n}, $\lambda_\star := \lambda_{\min}(H_\star)$ and $\lambda^\star := \lambda_{\max}(H_\star)$.
Let $\omega_\nu(\tau) := e^{\tau}$ if $\nu = 2$ and $(1 - \tau)^{-2/(\nu-2)}$ if $\nu > 2$.
\begin{align}\label{eq:R_nu_star}
    R_{\nu}^\star :=
    \begin{cases}
        \lambda_{\star}^{-1/2} R & \mbox{if } \nu = 2 \\
        (\nu/2 -1) \lambda_{\star}^{(\nu - 3)/2} R & \mbox{if } \nu \in (2, 3] \\
        (\nu/2 - 1) (\lambda^\star)^{(\nu - 3)/2} R & \mbox{if } \nu > 3.
    \end{cases}
\end{align}

\begin{customprop}{\ref{prop:emp_hess_est}}
    Fix $\varepsilon \in (0, K_\nu]$ and let $s_n := t_n\big( 3^{-d}[1.5 \omega_\nu(\varepsilon) n]^{d(1-\nu/2)} \delta/2 \big)$.
    Under \Cref{asmp:self_concordance,asmp:bernstein} with $r = K_\nu  / R_\nu^\star$, it holds that, with probability at least $1 - \delta$,
    \begin{align*}
        \frac1{2\omega_\nu^2(\varepsilon)} H_\star \preceq \frac{1-s_n}{\omega_\nu^2(\varepsilon)} H_\star \preceq H_n(\theta) \preceq (1 + s_n) \omega_\nu^2(\varepsilon) H_\star \preceq \frac32 \omega_\nu^2(\varepsilon) H_\star, \;\mbox{for all } \theta \in \Theta_{\varepsilon/R_\nu^\star}(\theta_\star),
    \end{align*}
    whenever $n \ge 4(K_2 + 2\sigma_H^2) \left\{ \log{(4d/\delta)} + d \log{[3(1.5 \omega_\nu(\varepsilon)n)^{\nu/2-1}]}\right\}$.
\end{customprop}

\begin{proof}
    We prove the result in the following steps.

    \emph{Step 1. Take a $\tau$-covering and relate $H_n(\theta)$ to $H_n(\bar \theta)$ for some $\bar \theta$ in the covering.}
    Let $\tau := \varepsilon / R_\nu^\star [1.5 \omega_\nu(\varepsilon) n]^{\nu/2-1}$. Take an $\tau$-covering $\calN_\tau$ of $\Theta_{\varepsilon/R_\nu^\star}(\theta_\star)$ w.r.t.~$\norm{\cdot}_{H_\star}$, and let $\pi(\theta)$ be the projection of $\theta$ onto $\calN_\tau$.
    Let
    \begin{align*}
        d_{n, \nu}(\theta_1, \theta_2) :=
        \begin{cases}
            n^{\nu/2-1} R \norm{\theta_2 - \theta_1}_2 & \mbox{if } \nu = 2 \\
            (\nu/2-1) n^{(\nu/2-1)} R \norm{\theta_2 - \theta_1}_2^{3 - \nu} \norm{\theta_2 - \theta_1}_{H_n(\theta_1)}^{\nu-2} & \mbox{otherwise}.
        \end{cases}
    \end{align*}
    By \Cref{lem:emp_risk_self_concordance} and \Cref{prop:hessian}, we have, for all $\theta \in \Theta_{\varepsilon/R_\nu^\star}(\theta_\star)$,
    \begin{align}\label{eq:emp_hess_sandwich}
        \frac1{\omega_\nu(d_{n,\nu}(\pi(\theta), \theta))} H_n(\pi(\theta)) \preceq H_n(\theta) \preceq \omega_\nu(d_{n,\nu}(\pi(\theta), \theta)) H_n(\pi(\theta)),
    \end{align}
    where it holds if $d_{n,\nu}(\pi(\theta), \theta) < 1$ for the case $\nu > 2$.
    
    \emph{Step 2. Relate $H_n(\theta)$ to $H_\star$ for all $\theta$ in the covering.} Fix an arbitrary $\theta \in \calN_\tau$. Following the same argument as \Cref{lem:hessian}, we have, with probability at least $1 - \delta$,
    \begin{align}\label{eq:concentration_Hn}
        (1-t_n) H(\theta) \preceq H_n(\theta) \preceq (1+t_n) H(\theta).
    \end{align}
    It follows from \Cref{asmp:self_concordance} and \Cref{lem:bound_d_nu} that
    \begin{align}\label{eq:pop_hess_bound}
        \frac1{\omega_\nu(R_\nu^\star \norm{\theta - \theta_\star}_{H_\star})} H_\star \preceq H(\theta) \preceq \omega_\nu(R_\nu^\star \norm{\theta - \theta_\star}_{H_\star}) H_\star,
    \end{align}
    since $R_\nu^\star \norm{\theta - \theta_\star}_{H_\star} \le \varepsilon \le K_\nu < 1$.
    By the monotonicity of $\omega_\nu$, we get
    \begin{align*}
        \frac1{\omega_\nu(\varepsilon)} H_\star \preceq H(\theta) \preceq \omega_\nu(\varepsilon) H_\star,
    \end{align*}
    and thus, with probability at least $1 - \delta$,
    \begin{align*}
        \frac{1-t_n(\delta/2)}{\omega_\nu(\varepsilon)} H_\star \preceq H_n(\theta) \preceq [1 + t_n(\delta/2)] \omega_\nu(\varepsilon) H_\star.
    \end{align*}
    Let $s_n := t_n\big( (\tau R_\nu^\star/3\varepsilon)^d \delta/2 \big)$ and
    \begin{align*}
        \calA := \left\{ \frac{1-s_n}{\omega_\nu(\varepsilon)} H_\star \preceq H_n(\pi(\theta)) \preceq (1+s_n) \omega_\nu(\varepsilon) H_\star, \mbox{ for all } \theta \in \Theta_{\varepsilon/R_{\nu}^\star}(\theta_\star) \right\}.
    \end{align*}
    Since $\abs{\calN_\tau} \le (3\varepsilon/\tau R_\nu^\star)^d$ \citep{ostrovskii2021finite}, by a union bound, we have $\Prob(\calA) \ge 1 - \delta$.
    
    \emph{Step 3. Combine the previous two steps.}
    On the event $\calA$, we have $H_n(\pi(\theta)) \preceq (1+s_n) \omega_\nu(\varepsilon) H_\star$ for all $\theta \in \Theta_{\varepsilon/R_\nu^\star}(\theta_\star)$.
    A similar argument as \Cref{lem:bound_d_nu} shows that
    \begin{align*}
        d_{n, \nu}(\pi(\theta), \theta) \le
        \begin{cases}
            \lambda_\star^{-1/2} R \tau & \mbox{if } \nu = 2 \\
            (\nu/2-1) \lambda_\star^{(\nu-3)/2} [(1+s_n)\omega_\nu(\varepsilon)]^{(\nu-2)/2} n^{\nu/2-1} R \tau & \mbox{if } \nu \in (2, 3] \\
            (\nu/2-1) (\lambda^\star)^{(\nu - 3)/2} [(1+s_n)\omega_\nu(\varepsilon)]^{(\nu-2)/2} n^{\nu/2-1} R \tau & \mbox{otherwise},
        \end{cases}
    \end{align*}
    which is equal to $[(1+s_n)\omega_\nu(\varepsilon)]^{\nu/2-1} n^{\nu/2-1}R_\nu^\star \tau$.
    When $n \ge 4(K_2 + 2\sigma_H^2) \left\{ \log{(4d/\delta)} + d \log{[3(1.5 \omega_\nu(\varepsilon)n)^{\nu/2-1}]}\right\}$, we have $s_n \le 1/2$, and thus
    substituting $\tau$ gives $d_{n, \nu}(\pi(\theta), \theta) \le \varepsilon \le K_\nu < 1$.
    Hence, by \eqref{eq:emp_hess_sandwich}, we obtain
    \begin{align*}
        \frac{1-s_n}{\omega_\nu^2(\varepsilon)} H_\star \preceq H_n(\theta) \preceq (1+s_n) \omega_\nu^2(\varepsilon) H_\star, \quad \mbox{for all } \theta \in \Theta_{\varepsilon/R_\nu^\star}(\theta_\star).
    \end{align*}
    on the event $\calA$.
\end{proof}

We give below the precise version of \Cref{thm:risk_bound_generalized}.
Recall $K_\nu$ and $R_\nu^\star$ from \Cref{cor:K_nu} and \eqref{eq:R_nu_star}.
\begin{customthm}{\ref{thm:risk_bound_generalized}}
    Let $\nu \in [2, 3)$.
    Under \Cref{asmp:self_concordance,asmp:sub_gaussian,asmp:bernstein} with $r = 0$, we have,
    whenever
    \begin{align*}
        n \ge \max\left\{ 4(K_2 + 2\sigma_H^2) \log{(4d/\delta)}, C\left[\frac{(R_\nu^\star)^2 K_1^2 d_\star \log{(e/\delta)}}{K_\nu^2}\right]^{1/(3-\nu)} \right\},
    \end{align*}
    the empirical risk minimizer $\theta_n$ uniquely exists and satisfies, with probability at least $1 - \delta$,
    \begin{align*}
        \norm{\theta_n - \theta_\star}^2_{H_\star} \le  \frac{16 d_\star}{n} + C K_{1}^2 \log{(e / \delta)} \frac{\norm{\Omega_\star}_2}{n}.
    \end{align*}
\end{customthm}

\begin{proof}
    Similar to the proof of \Cref{prop:score}, we define two events
    \begin{align*}
        \begin{split}
            \calA := \left\{ \norm{S_n(\theta_\star)}_{H_\star^{-1}}^{2} \le \frac{d_\star}n + CK_{1}^2 \log{(2e / \delta)} \frac{\norm{\Omega_\star}_2}{n} \right\} \quad \mbox{and} \quad \calB := \left\{ \frac12 H_\star \preceq H_n(\theta_\star) \preceq \frac32 H_\star \right\}.
        \end{split}
    \end{align*}
    In the following, we let
    \begin{align*}
        n \gtrsim \max\left\{ 4(K_2 + 2\sigma_H^2) \log{(4d/\delta)}, \left[\frac{(R_\nu^\star)^2 K_1^2 d_\star \log{(e/\delta)}}{K_\nu^2}\right]^{1/(3-\nu)} \right\}.
    \end{align*}
    Following the same argument as \Cref{prop:score}, we have $\Prob(\calA \calB) \ge 1 - \delta$ and
    \begin{align*}
        \norm{\grad_n(\theta_\star)}_{H_n^{-1}(\theta_\star)}^2 \le \frac{2 d_\star}{n} + C K_{1}^2 \log{(e / \delta)} \frac{\norm{\Omega_\star}_2}{n} \le C K_{1}^2 \log{(e / \delta)} \frac{d_\star}{n}.
    \end{align*}
    Now, it suffices to prove, on the event $\calA \calB$,
    \begin{align*}
        \norm{\theta_n - \theta_\star}_{H_\star}^2 \le \frac{16 d_\star}{n} + C K_{1}^2 \log{(e / \delta)} \frac{\norm{\Omega_\star}_2}{n}.
    \end{align*}
    
    Recall $R_{n, \nu}^\star$ and $R_{\nu}^\star$ from \eqref{eq:R_n_nu_star} and \eqref{eq:R_nu_star}.
    It is straightforward to check that $R_{n, \nu}^\star \le \sqrt{2} n^{\nu/2-1} R_\nu^\star$ for all $\nu \in [2, 3]$.
    Consequently, it holds that
    \begin{align*}
        R_{n,\nu}^\star \norm{S_n(\theta_\star)}_{H_n^{-1}(\theta_\star)} \lesssim R_\nu^\star n^{(\nu - 3)/2} \sqrt{K_{1}^2 \log{(e / \delta)} d_\star} \le K_\nu
    \end{align*}
    since $n^{3-\nu} \gtrsim (R_\nu^\star)^2 K_1^2 \log{(e / \delta)} d_\star/ K_\nu^2$.
    As a result, by \Cref{prop:localization}, we have that $\theta_n$ uniquely exists and satisfies
    \begin{align*}
        \norm{\theta_n - \theta_\star}_{H_n(\theta_\star)} \le 4 \norm{S_n(\theta_\star)}_{H_n^{-1}(\theta_\star)},
    \end{align*}
    and thus, using the event $\calB$,
    \begin{align*}
        \norm{\theta_n - \theta_\star}_{H_\star}^2
        \le 2\norm{\theta_n - \theta_\star}_{H_n(\theta_\star)}^2
        \le \frac{16 d_\star}{n} + C K_{1}^2 \log{(e / \delta)} \frac{\norm{\Omega_\star}_2}{n}.
    \end{align*}
\end{proof}

We give below the precise version of \Cref{thm:conf_set}.
\begin{customthm}{\ref{thm:conf_set}}
    Let $\nu \in [2, 3)$ and $r_n := \sqrt{CK_1^2 \log{(e/\delta)} d_\star / n}$.
    Suppose the same assumptions in \Cref{thm:risk_bound_generalized} hold true.
    Furthermore, suppose that \Cref{asmp:bernstein} holds with $r = K_\nu / R_\nu^\star$.
    Let
    \begin{align*}
        \calC_n(\delta) := \left\{\theta \in \Theta: \norm{\theta_n - \theta}_{H_n(\theta_n)}^2 \le 24 \omega_\nu^2(r_n R_\nu^\star) \frac{d_\star}{n} + C K_1^2 \omega_\nu^2(r_n R_\nu^\star) \log{(e/\delta)} \frac{\norm{\Omega_\star}_2}{n} \right\}.
    \end{align*}
    Then we have $\Prob(\theta_\star \in \calC_n(\delta)) \ge 1 - \delta$ whenever $n$ satisfies
    \begin{align*}
        n \ge C \max\left\{(K_2 + \sigma_H^2)\left[\log(2d/\delta) + d\log{(\omega_{\nu}(K_\nu) n)} \right], \left[ \frac{(R_\nu^\star)^2 K_1^2 d_\star \log{(e/\delta)}}{K_\nu^2} \right]^{1/(3-\nu)} \right\}.
    \end{align*}
    Here $C$ is an absolute constant which may change from line to line.
\end{customthm}

\begin{proof}
    We start by defining some events:
    \begin{align}\label{eq:def_events}
        \begin{split}
            \calA &:= \left\{ \norm{S_n(\theta_\star)}_{H_\star^{-1}}^{2} \le \frac{d_\star}n + C K_{1}^2 \log{(3e / \delta)} \frac{\norm{\Omega_\star}_2}{n} \right\} \\
            \calB &:= \left\{ \frac12 H_\star \preceq H_n(\theta_\star) \preceq \frac32 H_\star \right\} \\
            \calC &:= \left\{ \frac1{2 \omega_\nu^2(r_n R_\nu^\star)} H_\star \preceq H_n(\theta) \preceq \frac{3}{2} \omega_\nu^2(r_n R_\nu^\star) H_\star, \quad \mbox{for all } \theta \in \Theta_{r_n}(\theta_\star) \right\}.
        \end{split}
    \end{align}
    In the following, we let
    \begin{align*}
        n \ge C \max\left\{(K_2 + \sigma_H^2)\left[\log(2d/\delta) + d\log{(\omega_{\nu}(K_\nu) n)} \right], \left[ \frac{(R_\nu^\star)^2 K_1^2 d_\star \log{(e/\delta)}}{K_\nu^2} \right]^{1/(3-\nu)} \right\}.
    \end{align*}
    It then follows that $r_n R_{\nu}^\star \le K_\nu$.
    According to \Cref{lem:score}, \Cref{lem:hessian}, and \Cref{prop:emp_hess_est} (with $\varepsilon = r_n R_\nu^\star$), it holds that $\Prob(\calA) \ge 1 - \delta/3$, $\Prob(\calB) \ge 1 - \delta/3$, and $\Prob(\calC) \ge 1 - \delta/3$.
    This implies that $\Prob(\calA \calB \calC) \ge 1 - \delta$.
    Now, it suffices to prove, on the event $\calA \calB \calC$,
    \begin{align*}
        \norm{\theta_n - \theta_\star}_{H_n(\theta_n)}^2 \le 24 \omega_\nu^2(r_n R_\nu^\star) \frac{d_\star}{n} + C K_1^2 \omega_\nu^2(r_n R_\nu^\star) \log{(e/\delta)} \frac{\norm{\Omega_\star}_2}{n}.
    \end{align*}
    
    Following the same argument as \Cref{thm:risk_bound_generalized}, we obtain
    \begin{align*}
        \norm{\theta_n - \theta_\star}_{H_\star}^2
        \le 2\norm{\theta_n - \theta_\star}_{H_n(\theta_\star)}^2
        \le \frac{16d_\star}{n} + CK_1^2 \log{(e/\delta)} \frac{\norm{\Omega_\star}_2}{n}
        \le r_n^2.
    \end{align*}
    Therefore, using the event $\calC$, we have
    \begin{align*}
        \norm{\theta_n - \theta_\star}_{H_n(\theta_n)}^2
        \le \frac32 \omega_\nu^2(r_n R_\nu^\star) \norm{\theta_n - \theta_\star}_{H_\star}^2
        \le 24 \omega_\nu^2(r_n R_\nu^\star) \frac{d_\star}{n} + C K_1^2 \omega_\nu^2(r_n R_\nu^\star) \log{(e/\delta)} \frac{\norm{\Omega_\star}_2}{n},
    \end{align*}
    which completes the proof.
\end{proof}

\subsection{Consistency of $d_n$}
\label{sub:appendix:consist_dn}

Now we are ready to prove \Cref{prop:d_n}.
Recall $t_n$ from \eqref{eq:t_n} and $r_n$ from \Cref{thm:conf_set}.
\begin{customprop}{\ref{prop:d_n}}
    Let $\nu \in [2, 3)$ and $s_n := CMr_n + CK_1^2(1 + Mr_n) (d/n) \log{(Mnr_n/\delta)}$.
    Under Asms.~\ref{asmp:self_concordance}, \ref{asmp:subG_local}, \ref{asmp:bernstein}, and \ref{asmp:lip} with $r = K_\nu/R_\nu^\star$, it holds that, with probability at least $1 - \delta$,
    \begin{align*}
        \frac{1 - t_n}{\omega_\nu^2(r_n R_\nu^\star)(1 + s_n)} d_n \le d_\star \le \frac{(1 + t_n)\omega_\nu^2(r_n R_\nu^\star)}{1 - s_n} d_n
    \end{align*}
    whenever $n$ satisfies
    \begin{align*}
        n \ge C \max\left\{ (K_2 + \sigma_H^2 + K_1^2) \left[ \log{(2d/\delta)} + d\log{(\omega_{\nu}(K_\nu) n/\delta)} \right], \left[ (M + R_\nu^\star/K_\nu)^2 K_1^2 d_\star \log{(e/\delta)} \right]^{1/(3-\nu)} \right\}.
    \end{align*}
\end{customprop}

\begin{proof}
    Let $\tau := \delta / (Mn)$.
    Take a $\tau$-covering of $\calN_\tau$ of $\Theta_{r_n}(\theta_\star)$ w.r.t.~$\norm{\cdot}_{H_\star}$, and let $\pi(\theta)$ be the projection of $\theta$ onto $\calN_\tau$.
    For simplicity of the notation, we define
    \begin{align*}
        \norm{A}_{B} := \norm{B^{1/2} A B^{1/2}}
    \end{align*}
    for a symmetric matrix $A$ and a psd matrix $B$.
    We start by defining some events.
    Let
    \begin{equation*}
        \begin{split}
            \calA &:= \left\{ \norm{S_n(\theta_\star)}_{H_\star^{-1}}^{2} \lesssim \frac1n K_{1}^2 \log{(5e / \delta)} d_\star \right\} \\
            \calB &:= \left\{ (1-t_n) H_\star \preceq H_n(\theta_\star) \preceq (1+t_n) H_\star \right\} \\
            \calC &:= \left\{ \frac{1 - t_n}{\omega_\nu^2(r_n R_\nu^\star)} H_\star \preceq H_n(\theta) \preceq (1 + t_n) \omega_\nu^2(r_n R_\nu^\star) H_\star, \quad \mbox{for all } \theta \in \Theta_{r_n}(\theta_\star) \right\} \\
            \calD &:= \left\{ \sup_{\theta \in \Theta_{r_n}(\theta_\star)} \norm{G_n(\theta) - G_n(\pi(\theta))}_{G_\star^{-1}} \le 5M\tau/\delta \right\} \\
            \calE &:= \left\{ \sup_{\theta \in \Theta_{r_n}(\theta_\star)} \norm{G_n(\pi(\theta)) - G(\pi(\theta))}_{G_\star^{-1}} \lesssim K_1^2 (1 + Mr_n) h\left( \frac{d\log{(36 r_n/\tau)} + \log{(10/\delta)}}{n} \right) \right\},
        \end{split}
    \end{equation*}
    where $h(t) := \max\{t^2, t\}$.
    In the following, we let
    \begin{align}\label{eq:n_constraint}
        n \ge C \max\left\{ (K_2 + \sigma_H^2 + K_1^2) \left[ \log{(2d/\delta)} + d\log{(\omega_{\nu}(K_\nu) n/\delta)} \right], \left[ (M + R_\nu^\star/K_\nu)^2 K_1^2 d_\star \log{(e/\delta)} \right]^{1/(3-\nu)} \right\}.
    \end{align}
    It then follows that $t_n \le 1/2$, $r_n \le K_\nu / R_\nu^\star = r$ and $s_n < 1$.
    According to \Cref{lem:score}, \Cref{lem:hessian}, and \Cref{prop:emp_hess_est}, it holds that $\Prob(\calA) \ge 1 - \delta/5$, $\Prob(\calB) \ge 1 - \delta/5$, and $\Prob(\calC) \ge 1 - \delta/5$.
    In the following, we prove the claim in three steps.
    
    \emph{Step 1. Control the probability of $\calD$.}
    By Markov's inequality, it holds that
    \begin{align*}
        \Prob(\calD^c)
        &\le \frac{\delta}{5M\tau} \Expect\left[ \sup_{\theta \in \Theta_{r_n}(\theta_\star)} \norm{G_n(\theta) - G_n(\pi(\theta))}_{G_\star^{-1}} \right]
        \txtover{Jensen's}{\le} \frac{\delta}{5M\tau} \sup_{\theta \in \Theta_{r_n}(\theta_\star)} \norm{G(\theta) - G(\pi(\theta))}_{G_\star^{-1}}.
    \end{align*}
    According to \Cref{asmp:lip}, we have
    \begin{align}\label{eq:lip_reformulation}
        M \norm{\theta_1 - \theta_2}_{H_\star} \ge \Expect[\norm{G(\theta_1; Z) - G(\theta_2; Z)}_{G_\star^{-1}}], \quad \mbox{for all } \theta_1, \theta_2 \in \Theta_r(\theta_\star).
    \end{align}
    It follows from Jensen's inequality that
    \begin{align}\label{eq:lip_jensen}
        M \norm{\theta_1 - \theta_2}_{H_\star} \ge \norm{G(\theta_1) - G(\theta_2)}_{G_\star^{-1}}, \quad \mbox{for all } \theta_1, \theta_2 \in \Theta_r(\theta_\star).
    \end{align}
    As a result,
    \begin{align*}
        \Prob(\calD^c) \le \frac{\delta}{5\tau} \norm{\theta - \pi(\theta)}_{H_\star} \le \frac{\delta}{5}.
    \end{align*}
    
    \emph{Step 2. Control the probability of $\calE$.}
    According to \citet[Exercise 4.4.3]{vershynin2018high}, we have
    \begin{align}\label{eq:matrix_norm_cover}
        \norm{G_n(\pi(\theta)) - G(\pi(\theta))}_{G_\star^{-1}} \le \frac12 \sup_{v \in \calV_{1/4}} \abs{v^\top G_\star^{-1/2} [G_n(\pi(\theta)) - G(\pi(\theta))] G_\star^{-1/2} v},
    \end{align}
    where $\calV_{1/4}$ is a $1/4$-covering of the unit ball in $\reals^d$.
    Note that
    \begin{align*}
        v^\top G_\star^{-1/2} (G_n(\pi(\theta)) - G(\pi(\theta))) G_\star^{-1/2} v = \frac1n \sum_{i=1}^n [W_i - \Expect[W_i]],
    \end{align*}
    where $W_i := [v^\top G_\star^{-1/2} S(\pi(\theta); Z_i)]^2$.
    Let $\bar v := G(\pi(\theta))^{1/2} G_\star^{-1/2} v$.
    By \Cref{asmp:subG_local},
    \begin{align*}
        \norm{v^\top G_\star^{-1/2} S(\pi(\theta); Z_i)}_{\psi_2}
        &= \norm{\bar v^\top G(\pi(\theta))^{-1/2} S(\pi(\theta); Z_i)}_{\psi_2} \\
        &\le \norm{\bar v}_2 K_1 \le \norm{G(\pi(\theta))^{1/2} G_\star^{-1/2}} K_1.
    \end{align*}
    Since $\pi(\theta) \in \Theta_{r_n}(\theta_\star) \subset \Theta_{r}(\theta_\star)$, it follows from \eqref{eq:lip_jensen} that 
    \begin{align*}
        \norm{G_\star^{-1/2} G(\pi(\theta)) G_\star^{-1/2} - I_d} \le M \norm{\pi(\theta) - \theta_\star}_{H_\star} \le M r_n.
    \end{align*}
    and thus
    \begin{align*}
        \norm{v^\top G_\star^{-1/2} S(\pi(\theta); Z_i)}_{\psi_2} \le \sqrt{1 + Mr_n} K_1.
    \end{align*}
    This implies, by \citet[Lemma 2.7.6]{vershynin2018high}, $W_i$ is sub-Exponential with $\norm{W_i}_{\psi_1} \le K_1^2 (1 + Mr_n)$.
    It then follows from the Bernstein inequality that
    \begin{align*}
        \Prob\left( \abs{\frac1n \sum_{i=1}^n [W_i - \Expect[W_i]]} > t \right) \le 2 \exp\left(- c \min\left\{\frac{t^2}{K_1^4 (1 + Mr_n)^2}, \frac{t}{K_1^2 (1 + Mr_n)} \right\} \right).
    \end{align*}
    Since $\abs{\calN_\tau} \le (3r_n/\tau)^d$ and $\calV_{1/4} \le 12^d$,
    by a union bound, we get
    \begin{align*}
        &\quad \Prob\left( \frac12 \sup_{\theta \in \Theta_{r_n}(\theta_\star)} \sup_{v \in \calV_{1/4}} \abs{v^\top G_\star^{-1/2} [G_n(\pi(\theta)) - G(\pi(\theta))] G_\star^{-1/2} v} > t \right) \\
        &\le 2 \abs{\calN_\tau} \abs{\calV_{1/4}} \exp\left(- c \min\left\{\frac{4t^2}{K_1^4 (1 + Mr_n)^2}, \frac{2t}{K_1^2 (1 + Mr_n)} \right\} \right) \\
        &\le 2 (36r_n/\tau)^d \exp\left(- c \min\left\{\frac{4t^2}{K_1^4 (1 + Mr_n)^2}, \frac{2t}{K_1^2 (1 + Mr_n)} \right\} \right).
    \end{align*}
    Hence, it follows from \eqref{eq:matrix_norm_cover} that $\Prob(\calE^c) \le \delta/5$.
    
    \emph{Step 3. Prove the bound on the event $\calA \calB \calC \calD \calE$.}
    Following the same argument as \Cref{thm:risk_bound_generalized}, we obtain
    \begin{align}\label{eq:local}
        \norm{\theta_n - \theta_\star}_{H_\star} \lesssim \norm{\theta_n - \theta_\star}_{H_n(\theta_\star)} \lesssim n^{-1/2} \sqrt{K_{1}^2 \log{(e / \delta)} d_\star} = r_n.
    \end{align}
    Using the event $\calC$, we have
    \begin{align*}
        \frac1{(1 + t_n)\omega_\nu^2(r_n R_\nu^\star)} H_n(\theta_n) \preceq H_\star \preceq \frac{\omega_\nu^2(r_n R_\nu^\star)}{1 - t_n} H_n(\theta_n),
    \end{align*}
    and thus
    \begin{equation}\label{eq:d_star_first_bound}
    \begin{split}
        d_\star &\le (1+t_n) \omega_\nu^2(r_n R_\nu^\star) \Tr\left( H_n(\theta_n)^{-1/2} G_\star H_n(\theta_n)^{-1/2} \right) \\
        d_\star &\ge \frac{1-t_n}{\omega_\nu^2(r_n R_\nu^\star)} \Tr\left( H_n(\theta_n)^{-1/2} G_\star H_n(\theta_n)^{-1/2} \right).
    \end{split}
    \end{equation}
    Now it remains to control
    \begin{align*}
        \norm{G_n(\theta_n) - G_\star}_{G_\star^{-1}} \le \norm{G(\theta_n) - G_\star}_{G_\star^{-1}} + \norm{G_n(\theta_n) - G(\theta_n)}_{G_\star^{-1}}.
    \end{align*}
    We first control $\norm{G(\theta_n) - G_\star}_{G_\star^{-1}}$.
    It follows from \eqref{eq:lip_jensen} and \eqref{eq:local} that
    \begin{align*}
        \norm{G(\theta_n) - G_\star}_{G_\star^{-1}}
        \le M \norm{\theta_n - \theta_\star}_{H_\star}
        \lesssim M r_n.
    \end{align*}
    We then control $\norm{G_n(\theta_n) - G(\theta_n)}_{G_\star^{-1}}$.
    By \eqref{eq:local}, we have
    \begin{align*}
        \norm{G_n(\theta_n) - G(\theta_n)}_{G_\star^{-1}}
        \le \sup_{\theta \in \Theta_{r_n}(\theta_\star)} \norm{G_n(\theta) - G(\theta)}_{G_\star^{-1}}.
    \end{align*}
    It then follows from the triangle inequality that
    \begin{align*}
        \sup_{\theta \in \Theta_{r_n}(\theta_\star)} \norm{G_n(\theta) - G(\theta)}_{G_\star^{-1}} \le A_1 + A_2 + A_3,
    \end{align*}
    where
    \begin{align*}
        A_1 &:= \sup_{\theta \in \Theta_{r_n}(\theta_\star)} \norm{G(\pi(\theta)) - G(\theta)}_{G_\star^{-1}} \\
        A_2 &:= \sup_{\theta \in \Theta_{r_n}(\theta_\star)} \norm{G_n(\pi(\theta)) - G(\pi(\theta))}_{G_\star^{-1}} \\
        A_3 &:= \sup_{\theta \in \Theta_{r_n}(\theta_\star)} \norm{G_n(\theta) - G_n(\pi(\theta))}_{G_\star^{-1}}.
    \end{align*}
    
    To control $A_1$, note that, for all $\theta \in \Theta_{r_n}(\theta_\star)$,
    \begin{align*}
        \norm{G(\pi(\theta)) - G(\theta)}_{G_\star^{-1}} \txtover{\eqref{eq:lip_jensen}}{\le} M \norm{\pi(\theta) - \theta}_{H_\star}
        \le M \tau.
    \end{align*}
    Consequently, we obtain $A_1 \le M \tau$.
    To control $A_2$, we use the event $\calE$ to obtain
    \begin{align*}
        A_2 \lesssim K_1^2 (1 + Mr_n) h\left( \frac{d\log{(36 r_n/\tau)} + \log{(10/\delta)}}{n} \right).
    \end{align*}
    To control $A_3$, we use the event $\calD$ to obtain $A_3 \le 5M\tau/\delta$.
    Therefore,
    \begin{align*}
        \norm{G_n(\theta_n) - G_\star}_{G_\star^{-1}}
        &\le CM r_n + M\tau + 5M\tau / \delta + CK_1^2 (1 + Mr_n) h\left( \frac{d\log{(36 r_n/\tau)} + \log{(10/\delta)}}{n} \right) \\
        &= CM r_n + \frac{5+\delta}{n} + CK_1^2 (1 + Mr_n) h\left( \frac{d\log{(36M n r_n/\delta)} + \log{(10/\delta)}}{n} \right).
    \end{align*}
    This yields that
    \begin{align*}
        (1 - s_n) G_\star \preceq G_n(\theta_n) \preceq (1 + s_n) G_\star,
    \end{align*}
    and thus
    \begin{align*}
        \frac{1 - t_n}{\omega_\nu^2(r_n R_\nu^\star)(1 + s_n)} d_n \le d_\star \le \frac{(1 + t_n)\omega_\nu^2(r_n R_\nu^\star)}{1 - s_n} d_n.
    \end{align*}
\end{proof}

\subsection{Effective dimension}
\label{sub:appendix:effect_dim}

To gain a better understanding on the effective dimension $d_\star$, we summarize it in \Cref{tab:decay} under different regimes of eigendecay, assuming that $G_\star$ and $H_\star$ share the same eigenvectors.

\begin{table}[t]
    \caption{Comparison between the effective dimension $d_\star$ and the parameter dimension $d$ in different regimes of eigendecays of $G_\star$ and $H_\star$ assuming they share the same eigenvectors.}
    \label{tab:decay}
    \centering
    \renewcommand{\arraystretch}{1.2}
    \begin{tabular}{lccccc}
        \addlinespace[0.4em]
        \toprule
        & \multicolumn{2}{c}{\textbf{Eigendecay}} & \multicolumn{2}{c}{\textbf{Dimension Dependency}} & \textbf{Ratio} \\
        & $G_\star$ & $H_\star$ & $d_\star$ & $d$ & $d_\star/d$ \\
        \midrule
        Poly-Poly & $i^{-\alpha}$ & $i^{-\beta}$ & $d^{(\beta - \alpha + 1) \vee 0}$ & $d$ & $d^{(\beta - \alpha)\vee(-1)}$ \\
        Poly-Exp & $i^{-\alpha}$ & $ e^{-\nu i}$ & $d^{1-\alpha} e^{\nu d}$ & $d$ & $d^{-\alpha} e^{\nu d}$ \\
        Exp-Poly & $e^{-\mu i}$ & $i^{-\beta}$ & $1$ & $d$ & $d^{-1}$ \\
        Exp-Exp & $e^{-\mu i}$ & $e^{-\nu i}$ & \begin{tabular}{@{}c@{}} \qquad $d$ \mbox{ if }  $\mu = \nu$ \\ \qquad $1$ \mbox{ if } $\mu > \nu$ \\ $e^{(\nu - \mu) d}$ \mbox{ if } $\mu < \nu$ \end{tabular} & $d$ & \begin{tabular}{@{}c@{}} $1$ \mbox{ if }  $\mu = \nu$ \\ $d^{-1}$ \mbox{ if } $\mu > \nu$ \\ $d^{-1} e^{(\nu - \mu) d}$ \mbox{ if } $\mu < \nu$ \end{tabular} \\
        \bottomrule
    \end{tabular}
\end{table}

\section{Examples and applications}
\label{sec:example}
We give the derivations for the examples considered in \Cref{sub:examples} and prove the results for goodness-of-fit testing in \Cref{sub:goodness}.

\subsection{Examples}
\label{sub:append:examples}

\begin{example}[Generalized linear models]
    Let $Z := (X, Y)$ be a pair of input and output, where $X \in \calX \subset \reals^\tau$ and $Y \in \calY \subset \reals$.
    Let $t: \calX \times \calY \rightarrow \reals^d$ and $\mu$ be a measure on $\calY$.
    Consider the statistical model
    \begin{align*}
        p(y \mid x) \sim \frac{\exp(\ip{\theta, t(x, y)})}{\int \exp(\ip{\theta, t(x, \bar y)}) \D \mu(\bar y)} \D \mu(y)
    \end{align*}
    with $\norm{t(x, Y)}_2 \le M$ a.s.~under $p(y \mid x)$ for all $x$.
    It induces the loss function
    \begin{align*}
        \score(\theta; z) := -\ip{\theta, t(x, y)} + \log{\int \exp(\ip{\theta, t(x, \bar y)}) \D \mu(\bar y)}.
    \end{align*}
    
    We first verify \Cref{asmp:self_concordance}, i.e., show that it is generalized self-concordant for $\nu = 2$ and $R = 2M$.
    We denote by $\Expect_{Y\mid x}$ the expectation w.r.t.~$p(y \mid x)$.
    Note that $\log{\int \ip{\theta, t(x, \bar y)} \D \mu(\bar y)}$ is the cumulant generating function.
    It follows from some computation that
    \begin{align*}
        D_\theta \score(\theta; z) [u] &= -\ip{u, t(x, y)} + \Expect_{Y\mid x}\ip{u, t(x, Y)} \\
        D_\theta^2 \score(\theta; z) [u, u] &= \Expect_{Y \mid x}[\ip{u, t(x, Y)}^2] - [\Expect_{Y \mid x}\ip{u, t(x, Y)}]^2 \\
        D_\theta^3 \score(\theta; z)[u, u, v] &= \Expect_{Y \mid x}[\ip{u, t(x, Y)}^2 \ip{v, t(x, Y}] - \Expect_{Y \mid x}[\ip{u, t(x, Y)}^2] \Expect_{Y \mid x}\ip{v, t(x, Y)} \\
        &\quad - 2\Expect[\ip{u, t(x, Y)} \ip{v, t(x, Y)}] \Expect[\ip{u, t(x, Y)}] - 2[\Expect\ip{u, t(x, Y)}]^2 \Expect\ip{v, t(x, Y)}.
    \end{align*}
    As a result,
    \begin{align*}
        \abs{D_\theta^3 \score(\theta; z)[u, u, v]}
        &= \abs{\Expect_{Y \mid x}\left\{ \left[ \ip{u, t(x, Y)} - \Expect_{Y\mid x}\ip{u, t(x, Y)} \right]^2 \left[ \ip{v, t(x, Y)} - \Expect_{Y\mid x}\ip{v, t(x, Y)} \right] \right\}} \\
        &\le 2M \norm{v}_2 \Expect_{Y \mid x}\left\{ \left[ \ip{u, t(x, Y)} - \Expect_{Y\mid x}\ip{u, t(x, Y)} \right]^2 \right\}, \quad \mbox{by } \norm{t(x, Y)}_2 \txtover{a.s.}{\le} M \\
        &= 2M \norm{v}_2 D_\theta^2 \ell(\theta; z)[u, u],
    \end{align*}
    which completes the proof.
    
    We then verify \Cref{asmp:sub_gaussian} and \Cref{asmp:subG_local}.
    By \Cref{lem:bounded_subG}, it suffices to show that $\norm{S(\theta_\star; Z)}_2$ is a.s.~bounded.
    In fact,
    \begin{align*}
        S(\theta_\star; z) = -t(x, y) + \Expect_{p_{\theta_\star}(Y \mid x)}[t(x, Y)].
    \end{align*}
    Since $\abs{t(X, Y)}_2 \txtover{a.s.}{\le} M$, we get $\norm{S(\theta_\star; Z)}_2 \txtover{a.s.}{\le} 2M$ and thus the claim follows.
    \Cref{asmp:subG_local} can be verified similarly.
    
    Next, we verify \Cref{asmp:bernstein}.
    According to \Cref{lem:bounded_bernstein}, it is enough to prove that $\norm{H(\theta; Z)}_2$ is a.s.~bounded.
    In fact,
    \begin{align*}
        H(\theta; z) = \Expect_{Y \mid x}[t(x, Y) t(x, Y)^\top] - \Expect_{Y \mid x}[t(x, Y)] \Expect_{Y \mid x}[t(x, Y)]^\top.
    \end{align*}
    Since $\norm{t(X, Y) t(X, Y)^\top}_2 \le \norm{t(X, Y)}_2^2 \txtover{a.s.}{\le} M^2$, it follows that $\norm{H(\theta, Z)}_2 \txtover{a.s.}{\le} M^2$.
    
    Finally, we verify \Cref{asmp:lip}.
    It suffices to show that $\norm{G(\theta_1; Z) - G(\theta_2; Z)}_2 / \norm{\theta_1 - \theta_2}_2$ is a.s.~bounded.
    Note that
    \begin{align*}
        G(\theta_1; z) - G(\theta_2; z)
        &= \Expect_{p_{\theta_1}(Y \mid x)}[t(x, Y)] \Expect_{p_{\theta_1}(Y \mid x)}[t(x, Y)]^\top - \Expect_{p_{\theta_2}(Y \mid x)}[t(x, Y)] \Expect_{p_{\theta_2}(Y \mid x)}[t(x, Y)]^\top \\
        &\quad - 2t(x, y)\left\{ \Expect_{p_{\theta_1}(Y \mid x)}[t(x, Y)] - \Expect_{p_{\theta_2}(Y \mid x)}[t(x, Y)] \right\}^\top.
    \end{align*}
    For the second term, we have
    \begin{align*}
        \norm{-2t(x, y)\left\{ \Expect_{p_{\theta_1}(Y \mid x)}[t(x, Y)] - \Expect_{p_{\theta_2}(Y \mid x)}[t(x, Y)] \right\}^\top}_2
        \le 2\norm{t(x, y)}_2 \norm{\Expect_{p_{\theta_1}(Y \mid x)}[t(x, Y)] - \Expect_{p_{\theta_2}(Y \mid x)}[t(x, Y)]}_2.
    \end{align*}
    Note that
    \begin{align*}
        \Expect_{p_{\theta_1}(Y \mid x)}[t(x, Y)] - \Expect_{p_{\theta_2}(Y \mid x)}[t(x, Y)]
        &= \frac{\int t(x, y) \exp(\ip{\theta_1, t(x, y)}) \D \mu(y)}{\int \exp(\ip{\theta_1, t(x, y)}) \D \mu(y)} - \frac{\int t(x, y) \exp(\ip{\theta_2, t(x, y)}) \D \mu(y)}{\int \exp(\ip{\theta_2, t(x, y)}) \D \mu(y)}.
    \end{align*}
    Since $\abs{\ip{\theta, t(X, Y)}} \txtover{a.s.}{\le} [\norm{\theta - \theta_\star}_2 + \norm{\theta_\star}_2]M \le [\lambda_\star^{-1/2} r + \norm{\theta_\star}_2] M$ for all $\theta \in \Theta_{r}(\theta_\star)$, it holds that $\int \exp(\ip{\theta, t(X, y)}) \D \mu(y) \txtover{a.s.}{\ge} c$ for some $c > 0$ and $\theta \in \{\theta_1, \theta_2\}$.
    Now it remains to control
    \begin{align*}
        A_1 &:= \Big \Vert \int t(x, y) \exp(\ip{\theta_1, t(x, y)}) \D \mu(y) \int \exp(\ip{\theta_2, t(x, y)}) \D \mu(y) \\
        &\qquad - \int t(x, y) \exp(\ip{\theta_2, t(x, y)}) \D \mu(y) \int \exp(\ip{\theta_1, t(x, y)}) \D \mu(y) \Big \Vert_2.
    \end{align*}
    By the triangle inequality, we get $A_1 \le B_1 + B_2$ where
    \begin{align*}
        B_1 &:= \norm{\left[\int t(x, y) \exp(\ip{\theta_1, t(x, y)}) \D \mu(y) - \int t(x, y) \exp(\ip{\theta_2, t(x, y)}) \D \mu(y) \right] \int \exp(\ip{\theta_2, t(x, y)}) \D \mu(y)}_2 \\
        B_2 &:= \norm{\int t(x, y) \exp(\ip{\theta_2, t(x, y)}) \D \mu(y) \left[ \int \exp(\ip{\theta_2, t(x, y)}) \D \mu(y) - \int \exp(\ip{\theta_1, t(x, y)}) \D \mu(y) \right]}_2.
    \end{align*}
    Since $\abs{\ip{\theta_2, t(X, Y)}}$ and $d$ is a.s.~bounded, 
\end{example}

\myparagraph{Remark}
As a special case, the negative log-likelihood of the softmax regression with $\calX \subset \{x \in \reals^\tau: \norm{x} \le M\}$ and $\calY = \{1, \dots, K\}$ is generalized self-concordant with $\nu = 2$ and $R = 2M$.
In fact, the statistical model of the softmax regression is
\begin{align*}
    p(y = k \mid x) \sim \frac{\exp{\ip{w_k, x}}}{\sum_{j=1}^K \exp{\ip{w_j, x}}}.
\end{align*}
Define $\theta^\top := (w_1^\top, \dots, w_K^\top)$ and $t(x, y)^\top := (0_\tau^\top, \dots, x^\top, \dots, 0_\tau^\top)$ whose elements from $(y-1)\tau + 1$ to $y\tau$ are given by $x^\top$ and 0 elsewhere.
Then we have
\begin{align*}
    p(y=k \mid x) \sim \frac{\exp{\ip{\theta, t(x, k)}}}{\sum_{y=1}^K \exp{\ip{\theta, t(x, y)}}}.
\end{align*}
The claim then follows from the example above and $\norm{t(x, Y)}_2 = \norm{x}_2 \le M$.

\myparagraph{Remark}
The conditional random fields \citep{lafferty2001conditional} also fall into the category of generalized linear models.
For simplicity, we consider a conditional random field on a chain, i.e., for $x = (x_t)_{t=1}^T$ and $y = (y_t)_{t=1}^T$,
\begin{align*}
    p(y \mid x) \propto \exp\left\{ \sum_{t=1}^{T-1} \lambda_t f_t(x, y_t, y_{t+1}) + \sum_{t=1}^T \mu_t g_t(x, y_t) \right\}.
\end{align*}
Define $\theta^\top := (\lambda_1, \dots, \lambda_{T-1}, \mu_1, \dots, \mu_T)$ and
\begin{align*}
    t(x, y)^\top := \left( f_1(x, y_1, y_2), \dots, f_{T-1}(x, y_{T-1}, y_T), g_1(x, y_1), \dots, g_T(x, y_T) \right).
\end{align*}
Then we have
\begin{align*}
    p(y \mid x) \sim \frac{\exp\ip{\theta, t(x, y)}}{\int \exp\ip{\theta, t(x, \bar y)} \D \bar y}.
\end{align*}

\begin{example}[Score matching with exponential families]
    Assume that $\bbZ = \reals^p$.
    Consider an exponential family on $\reals^d$ with densities
    \begin{align*}
        \log{p_\theta(z)} = \theta^\top t(z) + h(z) - \Lambda(\theta).
    \end{align*}
    The non-normalized density $q_\theta$ then reads $\log{q_\theta(z)} = \theta^\top t(z) + h(z)$.
    As a result, the score matching loss becomes
    \begin{align*}
        \score(\theta; z) &= \sum_{k=1}^p \left[ \theta^\top \frac{\partial^2 t(z)}{\partial z_k^2} + \frac{\partial^2 h(z)}{\partial z_k^2} + \frac12 \left( \theta^\top \frac{\partial t(z)}{\partial z_k} + \frac{\partial h(z)}{\partial z_k} \right)^2 \right] + \text{const} \\
        &= \frac12 \theta^\top A(z) \theta - b(z)^\top \theta + c(z) + \text{const},
    \end{align*}
    where $A(z) := \sum_{k=1}^p \frac{\partial t(z)}{\partial z_k} \big(\frac{\partial t(z)}{\partial z_k}\big)^\top$ is p.s.d, $b(z) := \sum_{k=1}^p \left[ \frac{\partial^2 t(z)}{\partial z_k^2} + \frac{\partial h(z)}{\partial z_k} \frac{\partial t(z)}{\partial z_k} \right]$, and $c(z) := \sum_{k=1}^p \left[ \frac{\partial^2 h(z)}{\partial z_k^2} + \big(\frac{\partial h(z)}{\partial z_k}\big)^2 \right]$.
    Therefore, the score matching loss $\score(\theta; z)$ is convex.
    Moreover, since the third derivatives of $\ell(\cdot; z)$ is zero, the score matching loss is generalized self-concordant for all $\nu \ge 2$ and $R \ge 0$.
    When the true distribution $\Prob$ is supported on the non-negative orthant $\reals^p_+$, the score matching loss does not apply.
    Fortunately, a generalized score matching \citep{hyvarinen2007some,yu2019generalized} loss can be used to address this issue.
    Let $w_1, \dots, w_m: \reals_+ \to \reals_+$ be functions that are absolutely continuous in every bounded sub-interval of $\reals_+$.
    Then the generalized score matching loss reads
    \begin{align}\label{eq:score_matching}
        \score(\theta; z) = \sum_{j=1}^d \left[ w_j'(z_j) \partial_j \log{q(z)} + w_j(z_j) \partial_{jj} \log{q(z)} + \frac12 w_j(z_j) (\partial_j \log{q(z)})^2 \right] + \text{const},
    \end{align}
    which consists of a weighted version of the original score matching loss with weights $\{w_j(x_j)\}_{j=1}^d$ (the last two terms in \eqref{eq:score_matching}) and an additional term (the first term in \eqref{eq:score_matching}).
    According to \cite[Theorem 5]{yu2019generalized}, the loss \eqref{eq:score_matching} admits a quadratic form:
    \begin{align*}
        \score(z, Q_\theta) = \frac12 \theta^\top \bar A(z) \theta - \bar b(z)^\top \theta + \bar c(z) + \text{const},
    \end{align*}
    where $\bar A(z)$ is p.s.d.
    Hence, it is generalized self-concordant.
    Note that a particular example is the pairwise graphical models studies in \citep{yu2016statistical,yu2020simultaneous}.
\end{example}

\begin{example}[Generalized score matching with exponential families]
    When the true distribution $\Prob$ is supported on the non-negative orthant, $\reals^d_+$, the Hyv\"arinen score does not apply.
    Hyv\"arinen \citep{hyvarinen2007some} proposed the non-negative score matching to address this issue, which is later generalized in \cite[Section 2.2]{yu2019generalized}.
    Let $h_1, \dots, h_m: \reals_+ \to \reals_+$ be positive functions that are absolutely continuous in every bounded sub-interval of $\reals_+$.
    Then the generalized Hyv\"arinen score reads
    \begin{align}\label{eq:general_score_matching}
        \score(z, Q) = \sum_{j=1}^d \left[ h_j'(z_j) \partial_j \log{q(z)} + h_j(z_j) \partial_{jj} \log{q(z)} + \frac12 h_j(z_j) (\partial_j \log{q(z)})^2 \right],
    \end{align}
    which is a weighted version of the original Hyv\"arinen score with weights $\{h_j(x_j)\}_{j=1}^d$ (the last two terms in \eqref{eq:general_score_matching}) with an additional term (the first term in \eqref{eq:general_score_matching}).
    
    We then consider an exponential family on $\reals_+^d$ with densities
    \begin{align*}
        \log{q_\theta(z)} = \theta^\top t(z) - S(\theta) + b(z).
    \end{align*}
    According to \cite[Theorem 5]{yu2019generalized}, the score \eqref{eq:general_score_matching} admits the quadratic form:
    \begin{align*}
        \score(z, Q_\theta) = \frac12 \theta^\top \Gamma(z) \theta - g(z)^\top \theta + C,
    \end{align*}
    where $\Gamma(z)$ is p.s.d.
    Hence, this score is self-concordant.
    Note that a particular example is the pairwise graphical models studies in \citep{yu2016statistical,yu2020simultaneous}.
\end{example}

\subsection{Applications to goodness-of-fit testing}
\label{sub:appendix:application}

Before we start, we note that a simple modification to the confidence bound in \Cref{thm:conf_set} leads to the following risk bound that can be utilized to analyze the likelihood ratio test.

\begin{corollary}\label{cor:risk_bound}
    Under the same assumptions in \Cref{thm:conf_set}, we have, with probability at least $1 - \delta$,
    \begin{align*}
        L(\theta_n) - L(\theta_\star) \lesssim K_{1}^2 \omega_\nu^2(\varepsilon) \log{(e/\delta)} \frac{d_\star}{n}
    \end{align*}
    whenever $n$ satisfies \eqref{eq:n_large_enough}.
\end{corollary}
\begin{proof}
    By Taylor's expansion, we have
    \begin{align*}
        L(\theta_n) - L(\theta_\star) = S(\theta_\star)^\top (\theta_n - \theta_\star) + \frac12 \norm{\theta_n - \theta_\star}_{H_n(\bar \theta_n)}^2
    \end{align*}
    for some $\bar \theta_n \in \mbox{Conv}\{\theta_n, \theta_\star\} \subset \Theta_{\varepsilon / R_\nu^\star}(\theta_\star)$.
    By $S(\theta_\star) = 0$ and \Cref{thm:conf_set}, we get
    \begin{align*}
        L(\theta_n) - L(\theta_\star) \lesssim K_1^2 \omega_\nu^2(\varepsilon) \log{(e/\delta)} \frac{d_\star}{n}.
    \end{align*}
\end{proof}

We begin with the type I error rates of Rao's score test, the likelihood ratio test, and the Wald test.
Note that $d_\star = d$ under $\hnull$.

\begin{customprop}{\ref{prop:typeI}}
    Suppose that \Cref{asmp:sub_gaussian,asmp:bernstein} with $r = 0$ hold true.
    Under $\hnull$, we have, with probability at least $1 - \delta$,
    \begin{align*}
        \rao \lesssim K_1^2 \log{(e/\delta)} \frac{d}n
    \end{align*}
    whenever $n \ge 4(K_2 + 2\sigma_H^2) \log{(4d/\delta)}$.
    Furthermore, if \Cref{asmp:self_concordance,asmp:bernstein} with $r = K_\nu/R_\nu^\star$ hold true, we have, with probability at least $1 - \delta$,
    \begin{align*}
        \lr, \wald \lesssim K_1^2 \omega_\nu^2(\varepsilon) \log{(e/\delta)} \frac{d}{n}
    \end{align*}
    whenever $n$ satisfies \eqref{eq:n_large_enough}.
\end{customprop}

\begin{proof}
    Under $\hnull$, we have $\theta_\star = \theta_0$.
    It then follows from \Cref{prop:score} that, with probability at least $1 - \delta$,
    \begin{align*}
        \rao := \norm{S_n(\theta_0)}_{H_n^{-1}(\theta_0)}^2 \lesssim \frac1n K_1^2 \log{(e/\delta)} d
    \end{align*}
    whenever $n \ge 4(K_2 + 2\sigma_H^2) \log{(4d/\delta)}$.
    
    By Taylor's theorem, there exists $\bar \theta_n \in \mbox{Conv}\{\theta_n, \theta_\star\}$ such that
    \begin{align*}
        \lr = 2S_n^\top(\theta_n) (\theta_0 - \theta_n) + \norm{\theta_0 - \theta_n}_{H_n(\bar \theta_n)}^2 = \norm{\theta_0 - \theta_n}_{H_n(\bar \theta_n)}^2.
    \end{align*}
    Following a similar argument as \Cref{thm:conf_set}, we obtain, with probability at least $1 - \delta$,
    \begin{align*}
        \lr \lesssim K_1^2 \omega_\nu^2(\varepsilon) \log{(e/\delta)} \frac{d}{n}
    \end{align*}
    whenever $n$ satisfies \eqref{eq:n_large_enough}.
    The statement for $\wald$ follows directly from \Cref{thm:conf_set}.
\end{proof}

We then prove the result for statistical power given in \Cref{prop:power}.

\begin{customprop}{\ref{prop:power}}[Statistical power]
    Let $\theta_\star \neq \theta_0$ that may depend on $n$.
    The following statements are true for sufficiently large $n$.
    \begin{enumerate}
        \item[(a)] Suppose that $S(\theta_0) \neq 0$, $H(\theta_0) \succ 0$, and \Cref{asmp:self_concordance,asmp:sub_gaussian,asmp:bernstein} hold true with $r=0$.
        When $\theta_\star - \theta_0 = O(n^{-1/2})$ and $\tau_n := t_n(\alpha)/4 - \norm{S(\theta_0)}_{H(\theta_0)^{-1}}^2 - \Tr(\Omega(\theta_0))/n > 0$, we have
        \begin{align*}
            \Prob(\rao > t_n(\alpha)) &\le 2d \exp\left( - \frac{n}{4(K_2 + 2\sigma_H^2)} \right) + \exp\left( -c\min\left\{ \frac{n^2 \tau_n^2}{K_1^2 \norm{\Omega(\theta_0)}_2^2}, \frac{n\tau_n}{K_1 \norm{\Omega(\theta_0)}_\infty} \right\} \right).
        \end{align*}
        When $\theta_* - \theta_n = \omega(n^{-1/2})$, we have
        \begin{align*}
            \Prob(\rao > t_n(\alpha)) \ge 1 - 2d \exp\left( - \frac{n}{4(K_2 + 2\sigma_H^2)} \right) - \exp\left( -c\min\left\{ \frac{n^2 \bar \tau_n^2}{K_1^2 \norm{\Omega(\theta_0)}_2^2}, \frac{n\bar \tau_n}{K_1 \norm{\Omega(\theta_0)}_\infty} \right\} \right),
        \end{align*}
        where $\bar \tau_n := \left[ \norm{S(\theta_0)}_{H(\theta_0)^{-1}} - \sqrt{3t_n(\alpha)/4} \right]^2 - \Tr(\Omega(\theta_0))/n$.
        
        \item[(b)] Suppose that the assumptions in \Cref{thm:conf_set} hold true.
        When $\theta_\star - \theta_0 = O(n^{-1/2})$ and $\tau_n' := t_n(\alpha)/384 - \norm{\theta_\star - \theta_0}_{H(\theta_\star)}^2/64 - d/n > 0$, we have
        \begin{align*}
            \Prob(\lr > t_n(\alpha))
            \le \exp\left( -c\min\left\{ \frac{n^2 (\tau_n')^2}{K_1^2 \norm{\Omega(\theta_\star)}_2^2}, \frac{n\tau_n'}{K_1 \norm{\Omega(\theta_\star)}_\infty} \right\} \right) + \exp\left(-c \frac{\varepsilon^2 n^{3-\nu}}{(R_\nu^\star)^2 K_1^2 d} \right).
        \end{align*}
        When $\theta_* - \theta_n = \omega(n^{-1/2})$, we have
        \begin{align*}
            \Prob(\lr > t_n(\alpha)) \ge 1 - \exp\left( -c\min\left\{ \frac{n^2 (\bar \tau_n')^2}{K_1^2 \norm{\Omega(\theta_\star)}_2^2}, \frac{n\bar \tau_n'}{K_1 \norm{\Omega(\theta_\star)}_\infty} \right\} \right) - \exp\left(-c \frac{\varepsilon^2 n^{3-\nu}}{(R_\nu^\star)^2 K_1^2 d} \right),
        \end{align*}
        where
        \begin{align*}
            \bar \tau_n' := \left[ \norm{\theta_\star - \theta_0}_{H(\theta_\star)}/8 - \sqrt{t_n(\alpha)}/4 \right]^2 - d/n.
        \end{align*}
        
        \item[(c)] Suppose that the assumptions in \Cref{thm:conf_set} hold true.
        When $\theta_\star - \theta_0 = O(n^{-1/2})$ and $\tau_n' := t_n(\alpha)/384 - \norm{\theta_\star - \theta_0}_{H(\theta_\star)}^2/64 - d/n > 0$, we have
        \begin{align*}
            \Prob(\wald > t_n(\alpha))
            \le \exp\left( -c\min\left\{ \frac{n^2 (\tau_n')^2}{K_1^2 \norm{\Omega(\theta_\star)}_2^2}, \frac{n\tau_n'}{K_1 \norm{\Omega(\theta_\star)}_\infty} \right\} \right) + \exp\left(-c \frac{\varepsilon^2 n^{3-\nu}}{(R_\nu^\star)^2 K_1^2 d} \right).
        \end{align*}
        When $\theta_* - \theta_n = \omega(n^{-1/2})$, we have
        \begin{align*}
            \Prob(\wald > t_n(\alpha)) \ge 1 - \exp\left( -c\min\left\{ \frac{n^2 (\bar \tau_n')^2}{K_1^2 \norm{\Omega(\theta_\star)}_2^2}, \frac{n\bar \tau_n'}{K_1 \norm{\Omega(\theta_\star)}_\infty} \right\} \right) - \exp\left(-c \frac{\varepsilon^2 n^{3-\nu}}{(R_\nu^\star)^2 K_1^2 d} \right),
        \end{align*}
        where
        \begin{align*}
            \bar \tau_n' := \left[ \norm{\theta_\star - \theta_0}_{H(\theta_\star)}/8 - \sqrt{t_n(\alpha)}/4 \right]^2 - d/n.
        \end{align*}
    \end{enumerate}
\end{customprop}

\begin{proof}[Proof of \Cref{prop:power}]
    We are mostly interested in local alternatives, i.e., $\theta_\star \rightarrow \theta_0$ as $n \rightarrow \infty$.

    \textbf{Rao's score test}. Define four events
    \begin{align*}
        \calA &:= \{\rao > t_n(\alpha)\} \\
        \calB &:= \left\{ \frac12 H(\theta_0) \preceq H_n(\theta_0) \preceq \frac32 H(\theta_0) \right\} \\
        \calC &:= \left\{ 4 \norm{S_n(\theta_0) - S(\theta_0)}_{H(\theta_0)^{-1}}^2 > t_n(\alpha) - 4 \norm{S(\theta_0)}_{H(\theta_0)^{-1}}^2 \right\} \\
        \calD &:= \left\{ \norm{S_n(\theta_0) - S(\theta_0)}_{H(\theta_0)^{-1}} < \norm{S(\theta_0)}_{H(\theta_0)^{-1}} - \sqrt{3t_n(\alpha)/4} \right\}.
    \end{align*}
    Note that
    \begin{align*}
        S(\theta_0) = S(\theta_0) - S(\theta_\star) = H(\bar \theta) (\theta_0 - \theta_\star),
    \end{align*}
    where $\bar \theta \in \mbox{Conv}\{\theta_0, \theta_\star\}$.
    Due to \Cref{asmp:self_concordance}, we have
    \begin{align}\label{eq:hessian_order}
        e^{-R \norm{\bar \theta - \theta_0}_2} H(\theta_0) \preceq H(\bar \theta) \preceq e^{R \norm{\bar \theta - \theta_0}_2} H(\theta_0).
    \end{align}
    Therefore, we conclude that, as $n \rightarrow \infty$,
    \begin{align}\label{eq:score_order}
        S(\theta_0) = H(\bar \theta) (\theta_0 - \theta_\star) = \Theta(\theta_\star - \theta_0).
    \end{align}
    
    We first consider the case when $\theta_\star - \theta_0 = O(n^{-1/2})$.
    On the event $\calB$, it holds that
    \begin{align*}
        \rao
        \le 2 \norm{S_n(\theta_0)}_{H(\theta_0)^{-1}}^2
        \le 4 \norm{S_n(\theta_0) - S(\theta_0)}_{H(\theta_0)^{-1}}^2 + 4 \norm{S(\theta_0)}_{H(\theta_0)^{-1}}^2.
    \end{align*}
    This implies $\calA \calB \subset \calA \calC$ and thus
    \begin{align*}
        \Prob(\calA) = \Prob(\calA \calB) + \Prob(\calA \calB^c) \le \Prob(\calA \calC) + \Prob(\calB^c) \le \Prob(\calC) + \Prob(\calB^c)
    \end{align*}
    It follows from \Cref{thm:bernstein_matrix} that, when $n$ is large enough,
    \begin{align*}
        \Prob(\calB^c) \le 2d \exp\left( - \frac{n}{4(K_2 + 2\sigma_H^2)} \right).
    \end{align*}
    Moreover, note that $\calC = \{ \norm{S_n(\theta_0) - S(\theta_0)}_{H^{-1}(\theta_0)}^2 - \Tr(\Omega(\theta_0))/n \ge \tau_n \}$, where
    \begin{align*}
        \tau_n = t_n(\alpha)/4 - \norm{S(\theta_0)}_{H(\theta_0)^{-1}}^{2} - \Tr(\Omega(\theta_0))/n.
    \end{align*}
    By \Cref{thm:isotropic_tail}, we have, whenever $\tau_n > 0$,
    \begin{align*}
        \Prob(\calC) \le \exp\left( -c\min\left\{ \frac{n^2 \tau_n^2}{K_1^2 \norm{\Omega(\theta_0)}_2^2}, \frac{n\tau_n}{K_1 \norm{\Omega(\theta_0)}_\infty} \right\} \right).
    \end{align*}
    Consequently, it holds that, whenever $\tau_n > 0$ and $n$ is large enough,
    \begin{align*}
        \Prob(\calA) \le 2d \exp\left( - \frac{n}{4(K_2 + 2\sigma_H^2)} \right) + \exp\left( -c\min\left\{ \frac{n^2 \tau_n^2}{K_1^2 \norm{\Omega(\theta_0)}_2^2}, \frac{n\tau_n}{K_1 \norm{\Omega(\theta_0)}_\infty} \right\} \right).
    \end{align*}
    Note that, for large enough $n$, it holds that $\Tr(\Omega(\theta_0)) \rightarrow d$ and thus $t_n(\alpha) > \Tr(\Omega(\theta_0))/n$.
    Hence, it follows from \eqref{eq:score_order} that, as long as $\theta_\star - \theta_0 = o(n^{-1/2})$, $\tau_n > 0$ for sufficiently large $n$.
    
    We then consider the case when $\theta_\star - \theta_0 = \omega(n^{-1/2})$.
    On the event $\calB$, it holds that
    \begin{align*}
        \rao
        \ge 2 \norm{S_n(\theta_0)}_{H(\theta_0)^{-1}}^2 / 3
        \ge 4 [\norm{S(\theta_0)}_{H(\theta_0)^{-1}} -  \norm{S_n(\theta_0) - S(\theta_0)}_{H(\theta_0)^{-1}}]^2 / 3.
    \end{align*}
    By \Cref{thm:isotropic_tail}, it holds that $\norm{S_n(\theta_0) - S(\theta_0)}_{H(\theta_0)^{-1}} = O(n^{-1/2})$.
    By \eqref{eq:score_order}, we know that $\norm{S(\theta_0)}_{H(\theta_0)^{-1}} = \omega(n^{-1/2})$ and thus, for sufficiently large $n$,
    \begin{align*}
        \norm{S(\theta_0)}_{H(\theta_0)^{-1}} > \norm{S_n(\theta_0) - S(\theta_0)}_{H(\theta_0)^{-1}} + \sqrt{t_n(\alpha)}.
    \end{align*}
    This implies that $\calB \calD \subset \calA \calB$ and hence
    \begin{align*}
        \Prob(\calA) \ge \Prob(\calA \calB) \ge \Prob(\calB \calD) \ge 1 - \Prob(\calB^c) - \Prob(\calD^c).
    \end{align*}
    Following a similar argument as above, we have, whenever $n$ is large enough,
    \begin{align*}
        \Prob(\calA) \ge 1 - 2d \exp\left( - \frac{n}{4(K_2 + 2\sigma_H^2)} \right) - \exp\left( -c\min\left\{ \frac{n^2 \bar \tau_n^2}{K_1^2 \norm{\Omega(\theta_0)}_2^2}, \frac{n\bar \tau_n}{K_1 \norm{\Omega(\theta_0)}_\infty} \right\} \right),
    \end{align*}
    where
    \begin{align*}
        \bar \tau_n := \left[ \norm{S(\theta_0)}_{H(\theta_0)^{-1}} - \sqrt{3t_n(\alpha)/4} \right]^2 - \Tr(\Omega(\theta_0))/n.
    \end{align*}
    
    \textbf{The Wald test}.
    Notice that $d_\star = d$ since the model is well-specified.
    Fix $\varepsilon = \varepsilon_\nu$ so that $\omega_\nu(\varepsilon) \le 2$.
    Let $\delta := \exp\left(-c \frac{\varepsilon^2 n^{3-\nu}}{(R_\nu^\star)^2 K_1^2 d} \right)$.
    Define the following events
    \begin{align}\label{eq:events_wald}
        \begin{split}
            \calA &:= \left\{ \norm{S_n(\theta_\star)}_{H^{-1}(\theta_\star)}^{2} \le C K_{1}^2 \log{(e / \delta)} \frac{d}{n} \right\} \\
            \calB &:= \left\{ \frac12 H(\theta_\star) \preceq H_n(\theta_\star) \preceq \frac32 H(\theta_\star) \right\} \\
            \calC &:= \left\{ \frac1{2 \omega_\nu^2(\varepsilon)} H(\theta_\star) \preceq H_n(\theta) \preceq \frac{3}{2} \omega_\nu^2(\varepsilon) H(\theta_\star), \quad \mbox{for all } \theta \in \Theta_{\varepsilon/R_\nu^\star}(\theta_\star) \right\} \\
            \calD &:= \left\{ \wald > t_n(\alpha) \right\} \\
            \calE &:= \left\{ \norm{S_n(\theta_\star)}_{H(\theta_\star)^{-1}}^2 >  t_n(\alpha)/384 - \norm{\theta_\star - \theta_0}_{H(\theta_\star)}^2/64 \right\} \\
            \calF &:= \left\{ \norm{S_n(\theta_\star)}_{H(\theta_\star)^{-1}}^2 < \norm{\theta_\star - \theta_0}_{H(\theta_\star)}/8 - \sqrt{t_n(\alpha)}/4 \right\}.
        \end{split}
    \end{align}
    Following the proof of \Cref{thm:conf_set}, we get $\Prob(\calA \calB \calC) \ge 1 - \delta$ and,
    on the event $\calA \calB \calC$, we have, for sufficiently large $n$,
    \begin{align*}
        \frac1{4} H(\theta_\star) \preceq \frac1{2 \omega_\nu^2(\varepsilon)} H(\theta_\star) \preceq H_n(\theta_n) \preceq \frac32 \omega_\nu^2(\varepsilon) H(\theta_\star) \preceq 3 H(\theta_\star)
    \end{align*}
    and
    \begin{align}\label{eq:alt_local}
        \norm{\theta_n - \theta_\star}_{H(\theta_\star)} \le 4\sqrt{2} \norm{S_n(\theta_\star)}_{H_n(\theta_\star)^{-1}} \le 8 \norm{S_n(\theta_\star)}_{H(\theta_\star)^{-1}}.
    \end{align}
    
    We first consider the case $\theta_\star - \theta_0 = O(n^{-1/2})$.
    On the event $\calA \calB \calC$, it holds that
    \begin{align*}
        \norm{\theta_n - \theta_0}_{H_n(\theta_n)}^2
        &\le 3\norm{\theta_n - \theta_0}_{H(\theta_\star)}^2
        \le 6\norm{\theta_n - \theta_\star}_{H(\theta_\star)}^2 + 6\norm{\theta_\star - \theta_0}_{H(\theta_\star)}^2 \\
        &\le 384\norm{S_n(\theta_\star)}_{H_n(\theta_\star)^{-1}}^2 + 6\norm{\theta_\star - \theta_0}_{H(\theta_\star)}^2
    \end{align*}
    This implies that $\calA \calB \calC \calD \subset \calA \calB \calC \calE$ and thus
    \begin{align*}
        \Prob(\calD)
        = \Prob(\calA \calB \calC \calD) + \Prob((\calA \calB \calC)^c \calD)
        \le \Prob(\calE) + \Prob((\calA \calB \calC)^c).
    \end{align*}
    Moreover, note that $\calE = \{ \norm{S_n(\theta_0) - S(\theta_0)}_{H^{-1}(\theta_0)}^2 - d/n \ge \tau_n' \}$, where
    \begin{align*}
        \tau_n' = t_n(\alpha)/384 - \norm{\theta_\star - \theta_0}_{H(\theta_\star)}^2/64 - d/n.
    \end{align*}
    By \Cref{thm:isotropic_tail}, we have, whenever $\tau_n' > 0$,
    \begin{align*}
        \Prob(\calE) \le \exp\left( -c\min\left\{ \frac{n^2 (\tau_n')^2}{K_1^2 \norm{\Omega(\theta_\star)}_2^2}, \frac{n\tau_n'}{K_1 \norm{\Omega(\theta_\star)}_\infty} \right\} \right).
    \end{align*}
    Since $\Prob((\calA \calB \calC)^c) \le \delta = \exp\left(-c \frac{\varepsilon^2 n^{3-\nu}}{(R_\nu^\star)^2 K_1^2 d} \right)$, it holds that
    \begin{align*}
        \Prob(\calD)
        \le \exp\left( -c\min\left\{ \frac{n^2 (\tau_n')^2}{K_1^2 \norm{\Omega(\theta_\star)}_2^2}, \frac{n\tau_n'}{K_1 \norm{\Omega(\theta_\star)}_\infty} \right\} \right) + \exp\left(-c \frac{\varepsilon^2 n^{3-\nu}}{(R_\nu^\star)^2 K_1^2 d} \right).
    \end{align*}

    We then consider the case $\theta_\star - \theta_0 = \omega(n^{-1/2})$.
    On the event $\calA \calB \calC$, we get
    \begin{align*}
        \norm{\theta_n - \theta_0}_{H_n(\theta_n)}^2
        \ge \norm{\theta_n - \theta_0}_{H(\theta_\star)}^2 / 4
        \ge [\norm{\theta_\star - \theta_0}_{H(\theta_\star)} - \norm{\theta_n - \theta_\star}_{H(\theta_\star)}]^2 / 4.
    \end{align*}
    According to \eqref{eq:alt_local} and the event $\calA$, we have $\norm{\theta_n - \theta_\star}_{H(\theta_\star)} = O(n^{-1})$ and thus $\norm{\theta_n - \theta_\star}_{H(\theta_\star)} < \norm{\theta_\star - \theta_0}_{H(\theta_\star)}$ for sufficiently large $n$.
    As a result, it holds that
    \begin{align*}
        \norm{\theta_n - \theta_0}_{H_n(\theta_n)}^2
        \ge \left[ \norm{\theta_\star - \theta_0}_{H(\theta_\star)} - 8 \norm{S_n(\theta_\star)}_{H(\theta_\star)^{-1}} \right]^2/4.
    \end{align*}
    This implies that $\calA \calB \calC \calF \subset \calA \calB \calC \calD$ and thus
    \begin{align*}
        \Prob(\calD) \ge \Prob(\calA \calB \calC \calD) \ge \Prob(\calA \calB \calC \calF) \ge 1 - \Prob((\calA \calB \calC)^c) - \Prob(\calF^c).
    \end{align*}
    Let
    \begin{align*}
        \bar \tau_n' := \left[ \norm{\theta_\star - \theta_0}_{H(\theta_\star)}/8 - \sqrt{t_n(\alpha)}/4 \right]^2 - d/n.
    \end{align*}
    It is positive for sufficiently large $n$ since $\theta_\star - \theta_0 = \omega(n^{-1/2})$.
    By \Cref{thm:isotropic_tail} and $\Prob((\calA \calB \calC)^c) \le \delta$, it holds that
    \begin{align*}
        \Prob(\calD) \ge 1 - \exp\left( -c\min\left\{ \frac{n^2 (\bar \tau_n')^2}{K_1^2 \norm{\Omega(\theta_\star)}_2^2}, \frac{n\bar \tau_n'}{K_1 \norm{\Omega(\theta_\star)}_\infty} \right\} \right) - \exp\left(-c \frac{\varepsilon^2 n^{3-\nu}}{(R_\nu^\star)^2 K_1^2 d} \right).
    \end{align*}
    
    \textbf{The likelihood ratio test}.
    Note that
    \begin{align*}
        \ell_n(\theta_0) - \ell_n(\theta_n) = \norm{\theta_n - \theta_0}_{H_n(\bar \theta)}^2
    \end{align*}
    for some $\bar \theta \in \mbox{Conv}\{\theta_n, \theta_0\}$.
    The claim can be proved with the same argument as the one for the Wald test.
\end{proof}

\section{Technical tools}
\label{sec:tools}
In this section, we first recall and prove some key properties of generalized self-concordance.
We then review some key results regarding the concentration of random vectors and matrices.

\subsection{Properties of generalized self-concordant functions}
\label{sub:appendix:self_concordance}

Throughout this section, we let $f: \reals^d \rightarrow \reals$ be $(R, \nu)$-generalized self-concordant as in \Cref{def:general_self_concordance}, where $R > 0$ and $\nu \ge 2$.
For simplicity of the notation, we denote $\norm{\cdot}_x := \norm{\cdot}_{\nabla^2 f(x)}$.
Let
\begin{align}\label{eq:d_nu}
    d_\nu(x, y) :=
    \begin{cases}
        R \norm{y - x}_2 & \mbox{if } \nu = 2 \\
        (\nu/2 - 1) R \norm{y - x}_2^{3-\nu} \norm{y - x}_x^{\nu-2} & \mbox{if } \nu > 2
    \end{cases}
\end{align}
and
\begin{align}\label{eq:omega_nu}
    \omega_\nu(\tau) :=
    \begin{cases}
        (1 - \tau)^{-2/(\nu-2)} & \mbox{if } \nu > 2 \\
        e^{\tau} & \mbox{if } \nu = 2
    \end{cases}
\end{align}
with $\dom(\omega_\nu) = \reals$ if $\nu = 2$ and $\dom(\omega_\nu) = (-\infty, 1)$ if $\nu > 2$.

The next proposition gives bounds for the Hessian of $f$.
\begin{proposition}[\citet{sun2019generalized}, Prop.~8]
\label{prop:hessian}
    For any $x, y \in \dom(f)$, we have
    \begin{align*}
        \frac{1}{\omega_\nu(d_\nu(x, y))} \nabla^2 f(x) \preceq \nabla^2 f(y) \preceq \omega_\nu(d_\nu(x,y)) \nabla^2 f(x),
    \end{align*}
    where it holds if $d_\nu(x, y) < 1$ for the case $\nu > 2$.
\end{proposition}

We then give the bounds for function values.
Define two functions
\begin{align}\label{eq:bar_omega}
    \bar \omega_\nu(\tau) := \int_0^1 \omega_\nu(t \tau) \D t =
    \begin{cases}
        \tau^{-1} (e^\tau - 1) & \mbox{if } \nu = 2 \\
        -\tau^{-1} \log{(1 - \tau)} & \mbox{if } \nu = 4 \\
        \frac{\nu - 2}{\nu - 4} \frac{1 - (1 - \tau)^{(\nu-4)/(\nu-2)}}{\tau} & \mbox{otherwise}
    \end{cases}
\end{align}
and
\begin{align}
    \bbar{\omega}_\nu(\tau) := \int_0^1 t \bar \omega_\nu(t \tau) \D t =
    \begin{cases}
        \tau^{-2} (e^\tau - \tau - 1) & \mbox{if } \nu = 2 \\
        -\tau^{-2} [\tau + \log{(1 - \tau)}] & \mbox{if } \nu = 3 \\
        \tau^{-2} [(1 - \tau) \log{(1 - \tau)} + \tau] & \mbox{if } \nu = 4 \\
        \frac{\nu - 2}{\nu - 4} \frac1\tau \left[ \frac{\nu - 2}{2(3 - \nu) \tau} \left( (1 - \tau)^{2(3-\nu)/(2-\nu)} - 1 \right) - 1 \right] & \mbox{otherwise}.
    \end{cases}
\end{align}

\begin{proposition}[\citet{sun2019generalized}, Prop.~10]
\label{prop:function_value}
    For any $x, y \in \dom(f)$, we have
    \begin{align*}
        \bbar{\omega}_\nu(-d_\nu(x,y)) \norm{y - x}_x^2 \le f(y) - f(x) - \ip{\nabla f(x), y - x} \le \bbar{\omega}_\nu(d_\nu(x, y)) \norm{y - x}_x^2,
    \end{align*}
    where it holds if $d_\nu(x, y) < 1$ for the case $\nu > 2$.
\end{proposition}

In the following, we fix $x \in \dom(f)$ and assume $\nabla^2 f(x) \succ 0$.
We denote $\lambda_{\min} := \lambda_{\min}(\nabla^2 f(x))$ and $\lambda_{\max} := \lambda_{\max}(\nabla^2 f(x))$.
The next lemma bounds $d_\nu(x, y)$ with the local norm $\norm{y - x}_x$.
Let
\begin{align}\label{eq:R_nu}
    R_\nu :=
    \begin{cases}
        \lambda_{\min}^{-1/2} R & \mbox{if } \nu = 2 \\
        (\nu/2 -1) \lambda_{\min}^{(\nu - 3)/2} R & \mbox{if } \nu \in (2, 3] \\
        (\nu/2 - 1) \lambda_{\max}^{(\nu - 3)/2} R & \mbox{if } \nu > 3.
    \end{cases}
\end{align}
\begin{lemma}\label{lem:bound_d_nu}
    For any $\nu \ge 2$ and $y \in \dom(f)$, we have
    \begin{align}
        d_\nu(x, y) \le R_\nu \norm{y - x}_x.
    \end{align}
    Moreover, it holds that
    \begin{align*}
        \frac1{\omega_\nu(R_\nu \norm{y - x}_x)} \nabla^2 f(x) \preceq \nabla^2 f(y) \preceq \omega_\nu(R_\nu \norm{y - x}_x) \nabla^2 f(x),
    \end{align*}
    where it holds if $R_\nu \norm{y - x}_x < 1$ for the case $\nu > 2$.
\end{lemma}
\begin{proof}
    Recall the definition of $d_\nu$ in \eqref{eq:d_nu}.
    If $\nu = 2$, then, by the Cauchy-Schwarz inequality,
    \begin{align*}
        d_\nu(x, y) = R \norm{y - x}_2 \le \norm{[\nabla^2 f(x)]^{-1/2}}_2 R \norm{y - x}_x \le \lambda_{\min}^{-1/2} R \norm{y - x}_x.
    \end{align*}
    The case $\nu > 2$ can be proved similarly.
\end{proof}

We then prove some useful properties for the function $\bbar{\omega}$.
\begin{lemma}
\label{lem:monotonicity}
    For any $\nu \ge 2$, the following statements hold true:
    \begin{enumerate}[label=(\alph*)]
        \item\label{item:varphi} The function $\varphi(\tau) := \bbar{\omega}_\nu(-\tau)$ is strictly decreasing on $[0, \infty)$ with $\varphi(0) = 1/2$ and $\varphi(\tau) \ge 0$ for all $\tau \ge 0$.
        \item\label{item:psi} The function $\psi(\tau) := \bbar{\omega}_\nu(-\tau) \tau$ is strictly increasing on $[0, \infty)$ with $\psi(0) = 0$.
    \end{enumerate}
\end{lemma}
\begin{proof}
    \textbf{\ref{item:varphi}}.
    By definition, $\omega_\nu$ is strictly increasing on $(-\infty, 1)$.
    As a result, for any $\tau \in (-\infty, 1)$,
    \begin{align*}
        \bar \omega_\nu'(\tau) = \int_0^1 t \omega_\nu'(t \tau) \D t > 0.
    \end{align*}
    It then follows that, for any $\tau \ge 0$,
    \begin{align*}
        \varphi'(\tau)
        &= - \bbar{\omega}_\nu'(-\tau)
        = -\int_0^1 t^2 \bar \omega_\nu'(-t\tau) \D t < 0,
    \end{align*}
    and thus $\varphi$ is strictly decreasing on $[0, \infty)$.
    Note that $\omega_\nu(0) = 1$ and $\omega_\nu(\tau) > 0$ for all $\tau \in (-\infty, 1)$.
    It is straightforward to check that $\varphi(0) = 1/2$ and $\varphi(\tau) > 0$ for all $\tau \ge 0$.
    
    \textbf{\ref{item:psi}} Due to \eqref{eq:bar_omega}, it is clear that $\tau \mapsto \tau \bar \omega_\nu(-\tau)$ is strictly increasing on $[0, \infty)$ and equals 0 at $\tau = 0$.
    Note that, for any $\tau \ge 0$,
    \begin{align*}
        \psi(\tau) = \int_0^1 t\tau \bar \omega_\nu(-t\tau) \D t = \frac1\tau \int_0^\tau t \bar \omega_\nu(-t) \D t.
    \end{align*}
    We get
    \begin{align*}
        \psi'(\tau) = \frac1{\tau^2} \left[ \tau^2 \bar \omega_\nu(-\tau) - \int_0^\tau t \bar \omega_\nu(-t) \D t \right].
    \end{align*}
    By the monotonicity of $\tau \mapsto \tau \bar \omega_\nu(-\tau)$, it follows that $\psi'(\tau) > 0$.
\end{proof}

\begin{corollary}\label{cor:K_nu}
    Let $\tau \ge 0$.
    For any $\nu \ge 2$, there exists $K_\nu \in (0, 1/2]$ such that
    \begin{align*}
        \bbar{\omega}_\nu(-\tau) \tau \le K_\nu \Rightarrow \tau < 1 + \ind\{\nu = 2\} \mbox{ and } \bbar{\omega}_\nu(-\tau) \ge 1/4.
    \end{align*}
    In particular, $K_\nu = 1/2$ if $\nu = 2$ and $K_\nu = 1/4$ if $\nu = 3$.
\end{corollary}
\begin{proof}
    The existence of $K_\nu$ follows directly from the strict monotonicity of $\varphi$ and $\psi$ shown in \Cref{lem:monotonicity}.
    For $\nu = 2$,
    \begin{align*}
        \bbar{\omega}_\nu(-\tau) \tau = \frac{e^{-\tau} + \tau - 1}{\tau} \le 1/2 \Rightarrow \tau < 2.
    \end{align*}
    As a result, we have $\bbar{\omega}_\nu(-\tau) \ge 1/4$.
    The case for $\nu = 3$ can be proved similarly.
\end{proof}

The next result shows that the local distance between the minimizer of $f$ and $x$ only depends on the geometry at $x$.
It can be used to localize the empirical risk minimizer as in \Cref{prop:localization}.
\begin{proposition}
\label{prop:self_concordance_local}
    Whenever $R_\nu \norm{\nabla f(x)}_{\nabla^2 f(x)^{-1}} \le K_\nu$, the function $f$ has a unique minimizer $\bar x$ and
    \begin{align*}
        \norm{\bar x - x}_x \le 4 \norm{\nabla f(x)}_{\nabla^2 f(x)^{-1}}.
    \end{align*}
\end{proposition}
\begin{proof}
    Consider the level set
    \begin{align*}
        \calL_f(f(x)) := \{y \in \calX: f(y) \le f(x)\} \neq \emptyset.
    \end{align*}
    Take an arbitrary $y \in \calL_f(f(x))$.
    According to \Cref{prop:function_value}, we have
    \begin{align*}
        0 \ge f(y) - f(x) \ge \ip{\nabla f(x), y - x} + \bbar{\omega}_\nu(-d_\nu(x, y)) \norm{y-x}_x^2.
    \end{align*}
    By the Cauchy-Schwarz inequality and \Cref{lem:bound_d_nu,lem:monotonicity}, we get
    \begin{align*}
        \bbar{\omega}_\nu(-R_\nu \norm{y - x}_x) \norm{y - x}_x^2 \le \norm{\nabla f(x)}_{H^{-1}(x)} \norm{y - x}_x
    \end{align*}
    This implies
    \begin{align*}
        \bbar{\omega}_\nu(-R_\nu \norm{y - x}_x) R_\nu \norm{y - x}_x \le R_\nu \norm{\nabla f(x)}_{H^{-1}(x)} \le K_\nu.
    \end{align*}
    Due to \Cref{cor:K_nu}, it holds that $R_\nu \norm{y - x}_x < 1 + \ind\{\nu = 2\}$ and $\bbar{\omega}_\nu(-R_\nu \norm{y - x}_x) \ge 1/4$.
    It follows that $d_\nu(x, y) < 1 + \ind\{\nu = 2\}$ and
    \begin{align*}
        \norm{y - x}_x \le 4 \norm{\nabla f(x)}_{\nabla^2 f(x)^{-1}}.
    \end{align*}
    Hence, the level set $\calL_f(f(x))$ is compact so that $f$ has a minimizer $\bar x$.
    Moreover, by \Cref{prop:hessian} and $\nabla^2 f(x) \succ 0$, we obtain $\nabla^2 f(y) \succ 0$ for all $y \in \calL_f(f(x))$.
    This yields that $\bar x$ is the unique minimizer of $f$ and it satisfies
    \begin{align*}
        \norm{\bar x - x}_x \le 4 \norm{\nabla f(x)}_{\nabla^2 f(x)^{-1}}.
    \end{align*}
\end{proof}

\begin{remark}
    A similar result also appears in \cite[Prop.~B.4]{ostrovskii2021finite}.
    We extend their result from $\nu \in \{2, 3\}$ to $\nu \ge 2$.
\end{remark}

\subsection{Concentration of random vectors and matrices}
\label{sub:appendix:concentration}

We start with the precise definition of sub-Gaussian random vectors \cite[Chapter 3.4]{vershynin2018high}.

\begin{definition}[Sub-Gaussian vector]\label{def:subg_vec}
    Let $S \in \reals^d$ be a random vector.
    We say $S$ is sub-Gaussian if $\ip{S, s}$ is sub-Gaussian for every $s \in \reals^d$.
    Moreover, we define the sub-Gaussian norm of $S$ as
    \begin{align*}
        \norm{S}_{\psi_2} := \sup_{\norm{s}_2 = 1} \norm{\ip{S, s}}_{\psi_2}.
    \end{align*}
    Note that $\norm{\cdot}_{\psi_2}$ is a norm and satisfies, e.g., the triangle inequality.
\end{definition}

\begin{remark}\label{rmk:centering}
    When $S$ is not mean-zero, we have
    \begin{align*}
        \norm{S - \Expect[S]}_{\psi_2}
        = \sup_{\norm{s}_2 = 1} \norm{\ip{S - \Expect[S], s}}_{\psi_2}
        = \sup_{\norm{s}_2 = 1} \norm{s^\top S - \Expect[s^\top S]}_{\psi_2}.
    \end{align*}
    According to \citet[Lemma 2.6.8]{vershynin2018high}, we obtain
    \begin{align*}
        \norm{S - \Expect[S]}_{\psi_2} \le C \sup_{\norm{s}_2 = 1} \norm{s^\top S}_{\psi_2} = C \norm{S}_{\psi_2},
    \end{align*}
    where $C$ is an absolute constant.
\end{remark}

It follows from \citet[Eq.~(2.17)]{vershynin2018high} that a bounded random vector is sub-Gaussian.
\begin{lemma}\label{lem:bounded_subG}
    Let $S$ be a random vector such that $\norm{S}_2 \le M$ for some constant $M > 0$.
    Then $X$ is sub-Gaussian with $\norm{X}_{\psi_2} \le M/\sqrt{\log{2}}$.
\end{lemma}

As a direct consequence of \citet[Prop.~2.6.1]{vershynin2018high}, the sum of i.i.d.~sub-Gaussian random vectors is also sub-Gaussian.
\begin{lemma}\label{lem:sum_subg}
Let $S_1, \dots, S_n$ be i.i.d.~random vectors, then we have $\norm{\sum_{i=1}^n S_i}_{\psi_2}^2 \lesssim \sum_{i=1}^n \norm{S_i}_{\psi_2}^2$.
\end{lemma}

We call a random vector $S \in \reals^d$ isotropic if $\Expect[S] = 0$ and $\Expect[SS^\top] = I_d$.
The following theorem is a tail bound for quadratic forms of isotropic sub-Gaussian random vectors.
\begin{theorem}[\citet{ostrovskii2021finite}, Theorem A.1]\label{thm:isotropic_tail}
Let $S \in \reals^d$ be an isotropic random vector with $\norm{S}_{\psi_2} \le K$, and let $J \in \reals^{d \times d}$ be positive semi-definite.
Then,
\begin{align*}
    \Prob(\norm{S}_{J}^2 - \text{Tr}(J) \ge t) \le \exp\left(-c\min\left\{ \frac{t^2}{K^2 \norm{J}_2^2}, \frac{t}{K\norm{J}_\infty} \right\} \right).
\end{align*}
In other words, with probability at least $1 - \delta$, it holds that
\begin{align}
  \norm{S}_{J}^2 - \text{Tr}(J) \lesssim K^2\left( \norm{J}_2 \sqrt{\log{(e/\delta)}} + \norm{J}_{\infty} \log{(1/\delta)} \right).
\end{align}
\end{theorem}

We then give the definition of the matrix Bernstein condition \cite[Chapter 6.4]{wainwright2019high}.

\begin{definition}[Matrix Bernstein condition]\label{def:matrix_bernstein}
    Let $H \in \reals^{d \times d}$ be a zero-mean symmetric random matrix.
    We say $H$ satisfies a Bernstein condition with parameter $b > 0$ if, for all $j \ge 3$,
    \begin{align*}
        \Expect[H^j] \preceq \frac12 j! b^{j-2} \Var(H).
    \end{align*}
\end{definition}

The next lemma, which follows from \citet[Eq.~(6.30)]{wainwright2019high}, shows that a matrix with bounded spectral norm satisfies the matrix Bernstein condition.
\begin{lemma}\label{lem:bounded_bernstein}
    Let $H$ be a zero-mean random matrix such that $\norm{H}_2 \le M$ for some constant $M > 0$.
    Then $H$ satisfies the matrix Bernstein condition with $b = M$ and $\sigma_H^2 = \norm{\Var(H)}_2$.
    Moreover, $\sigma_H^2 \le 2M^2$.
\end{lemma}

The next theorem is the Bernstein bound for random matrices.

\begin{theorem}[\citet{wainwright2019high}, Theorem 6.17]\label{thm:bernstein_matrix}
Let $\{H_i\}_{i=1}^n$ be a sequence of zero-mean independent symmetric random matrices that satisfies the Bernstein condition with parameter $b > 0$.
Then, for all $\delta > 0$, it holds that
\begin{align}
  \Prob\left( \anorm{\frac1n \sum_{i=1}^n H_i}_2 \ge \delta \right) \le 2 \Rank\left(\sum_{i=1}^n \Var(H_i)\right) \exp\left\{ -\frac{n\delta^2}{2(\sigma^2 + b\delta)} \right\},
\end{align}
where $\sigma^2 := \frac1n \anorm{\sum_{i=1}^n \Var(H_i)}_2$.
\end{theorem}

\end{document}